\documentclass{article}
\usepackage[utf8]{inputenc}
\usepackage{amsmath,amsfonts}
\usepackage{graphicx}
\usepackage{bm}
\usepackage[usenames,dvipsnames]{xcolor}
\usepackage{algorithm}
\usepackage{mathrsfs}
\usepackage{bbm}
\usepackage{cases}
\usepackage{amsthm}
\usepackage{comment}
\usepackage{algpseudocode}
\usepackage{geometry}
\usepackage{tikz}
\usepackage{pgfplots}
\usetikzlibrary{intersections,shapes.arrows}
\usetikzlibrary{decorations.pathreplacing,decorations.markings}
\usetikzlibrary{arrows,decorations.markings}
\usetikzlibrary{spy}
\usetikzlibrary{pgfplots.groupplots}
\usetikzlibrary{decorations.pathreplacing,calligraphy}
\pgfplotsset{compat=1.14}
\numberwithin{equation}{section} 

\newtheorem{proposition}{Proposition}
\theoremstyle{remark}
\newtheorem*{remark}{Remark}

\usepackage[bbgreekl]{mathbbol}
\DeclareSymbolFontAlphabet{\mathbbm}{bbold}
\DeclareSymbolFontAlphabet{\mathbb}{AMSb}%

\newcommand{\Rbb}{\mathbb{R}}
\newcommand{\Pbb}{\mathbb{P}}

\newcommand{\Tcal}{\mathcal{T}}
\newcommand{\Hcal}{\mathcal{H}}

\newcommand{\Rcal}{\mathcal{R}}

\newcommand{\Scal}{\mathcal{S}}
\newcommand{\Ecal}{\mathcal{E}}

\newcommand{\abf}{\mathbf{a}}

\newcommand{\cbf}{\mathbf{c}}
\newcommand{\sbf}{\mathbf{s}}
\newcommand{\vbf}{\mathbf{v}}
\newcommand{\xbf}{\mathbf{x}}
\newcommand{\ybf}{\mathbf{y}}
\newcommand{\Abf}{\mathbf{A}}
\newcommand{\Vbf}{\mathbf{V}}

\newcommand{\Gbf}{\mathbf{G}}
\newcommand{\Rbf}{\mathbf{R}}

\newcommand{\pos}{\mathrm{x}}
\newcommand{\vel}{\mathrm{v}}
\newcommand{\posbf}{\mathbf{x}}
\newcommand{\velbf}{\mathbf{v}}
\newcommand{\dpos}{\dot{\mathbf{x}}}
\newcommand{\dvel}{\dot{\mathbf{v}}}

\newcommand{\rposbf}{\mathbf{y}}
\newcommand{\rvelbf}{\mathbf{w}}

\newcommand{\sbfi}{\mathrm{s}}

\newcommand{\fedim}{\kappa} 
\newcommand{\idx}{\ensuremath{\sigma}} 
\newcommand{\Gfun}{\Gbf}
\newcommand{\mGfun}{G} 
\newcommand{\jacG}[1]{\ensuremath{J_\mGfun^{#1}}}
\newcommand{\jacGhat}[1]{\ensuremath{J_{\widehat{\mGfun}}^{#1}}}
\newcommand{\deimb}[1]{\ensuremath{U^{#1}}} 
\newcommand{\deimi}[1]{\ensuremath{P^{#1}}} 
\newcommand{\jacE}{J_\mdiscref} 
\newcommand{\tsind}{\tau} 
\newcommand{\stmat}{T} 
\newcommand{\decd}{\mathbf{g}} 
\newcommand{\decdi}{\mathrm{g}} 
\newcommand{\ham}{H}
\newcommand{\nlham}{h} 
\newcommand{\rham}{\ham_{\rom}} 
\newcommand{\hrham}{\ham_{\hr}} 
\newcommand{\nlrham}{\nlham_{\rom}} 
\newcommand{\nlhrham}{\nlham_{\hr}} 
\newcommand{\redb}{\Psi} 
\newcommand{\coef}{Z} 
\newcommand{\cX}{Y} 
\newcommand{\cV}{W} 
\newcommand{\cXbf}{\mathbf{Y}} 
\newcommand{\cVbf}{\mathbf{W}} 
\newcommand{\cXs}{Y_{\star}} 
\newcommand{\cVs}{W_{\star}} 
\newcommand{\fX}{X} 
\newcommand{\fV}{V} 
\newcommand{\fXbf}{\mathbf{X}} 
\newcommand{\dfX}{\dot{\fX}}
\newcommand{\dfV}{\dot{\fV}}
\newcommand{\dcX}{\dot{\cX}}
\newcommand{\dcV}{\dot{\cV}}
\newcommand{\sol}{\Theta}
\newcommand{\dfos}{\dot{\sol}_{\fom}}
\newcommand{\fos}{\sol_{\fom}} 
\newcommand{\fostar}{\sol_{\fom,\star}} 
\newcommand{\fosvd}[1]{\sol_{\fom,#1}} 
\newcommand{\ros}{\sol_{\rom}} 
\newcommand{\rostar}{\sol_{\rom,\star}} 
\newcommand{\rX}{\fX_r} 
\newcommand{\rV}{\fV_r} 
\newcommand{\mdiscref}{E} 
\newcommand{\prm}{\eta} 
\newcommand{\retr}{\Rcal} 
\newcommand{\prpsi}{\Upsilon}
\newcommand{\resid}{R} 
\newcommand{\errappr}{\varepsilon} 
\newcommand{\merrappr}{\Ecal} 
\newcommand{\prms}{\Gamma} 
\newcommand{\spei}{p_I^{\star}} 
\newcommand{\spav}{p_A^{\star}} 
\newcommand{\spdb}{p_U^{\star}} 
\newcommand{\eimuf}{\delta} 
\newcommand{\prma}{\alpha} 
\newcommand{\prmsd}{\sigma} 
\newcommand{\wn}{k_0} 
\newcommand{\mhrdiscref}{\mdiscref_{\hr}}
\newcommand{\stm}{\text{St}(n,\Rbb^N)} 
\newcommand{\pdeg}{k} 
\newcommand{\solerr}{\Ecal} 
\newcommand{\hamerr}{\Ecal_H} 
\newcommand{\norm}[1]{\left\lVert#1\right\rVert}
\newcommand{\normF}[1]{\lVert#1\rVert_F}
\newcommand{\cone}{C_1} 
\newcommand{\ctwo}{C_2}
\newcommand{\errapprt}{\widetilde{\errappr}} 
\newcommand{\merrapprt}{\widetilde{\merrappr}}
\newcommand{\coeffint}{\cbf} 
\newcommand{\mcoeffint}{C} 
\newcommand{\sbfnt}{\xi} 
\newcommand{\msbfnt}{\Xi} 
\newcommand{\nti}{\nu} 
\newcommand{\errind}{\mathbb{E}} 
\newcommand{\avgsolerr}{\solerr^{\avg}} 
\newcommand{\Psbfnt}{\mathcal{P}}
\newcommand{\spmat}{\Pi}
\newcommand{\ns}{N_s} 
\newcommand{\nt}{N_t} 

\DeclareMathOperator{\hr}{hr}
\DeclareMathOperator{\old}{old}
\DeclareMathOperator{\new}{new}
\DeclareMathOperator{\fom}{f}
\DeclareMathOperator{\rom}{r}
\DeclareMathOperator{\avg}{avg}

\usepackage[
	style=numeric,
	citestyle=numeric,sorting=nyvt,sortcites=false,
	giveninits=true,
	maxnames=10,
	url=false,eprint=true,doi=false,isbn=false,
	backref=false,
	natbib=true,
]{biblatex}
\renewbibmacro{in:}{} 
\AtEveryBibitem{\clearfield{note}} 
\AtEveryBibitem{\ifentrytype{book}{\clearfield{pages}}{}}
\usepackage[
	bookmarks,
	colorlinks=true, 
	pdftitle={Adaptive HROM},
	pdfauthor={},
	pdfsubject={},
	pdfkeywords={},
	urlcolor= blue,
	linkcolor= blue,
	citecolor= blue
]{hyperref}
\usepackage{cleveref}
\usepackage{empheq}

\newcommand{\email}[1]{\protect\href{mailto:#1}{#1}}

\title{Adaptive hyper-reduction of non-sparse operators:\\
 application to parametric particle-based kinetic plasma models}
\author{Cecilia Pagliantini\thanks{Dipartimento di Matematica,
			  Universit\`a di Pisa, Italy.
  \email{cecilia.pagliantini@unipi.it}\\
  Funding from the MIUR Excellence Department Project awarded to the Department of Mathematics, University of Pisa, CUP I57G22000700001, and from the INDAM/GNCS 2024 project CUP E53C23001670001 are acknowledged.}
\and Federico Vismara\thanks{
\email{f.vismara@tue.nl}.}}
\date{December 22, 2025}

\begin{document}

\maketitle

\begin{abstract}
This paper proposes an adaptive hyper-reduction method to reduce the computational cost associated with the simulation of parametric particle-based kinetic plasma models, specifically focusing on the Vlasov-Poisson equation. Conventional model order reduction and hyper-reduction techniques are often ineffective for such models due to the non-sparse nature of the nonlinear operators arising from the interactions between particles.
To tackle this issue, we propose an adaptive, structure-preserving hyper-reduction method that leverages a decomposition of the discrete reduced Hamiltonian into a linear combination of terms, each depending on a few components of the state. The proposed approximation strategy allows to: (i) preserve the Hamiltonian structure of the problem; (ii) evaluate nonlinear non-sparse operators in a computationally efficient way; (iii) overcome the Kolmogorov barrier of transport-dominated problems via evolution of the approximation space and adaptivity of the rank of the solution. 
The proposed method is validated on numerical benchmark simulations, demonstrating stable and accurate performance with substantial runtime reductions compared to the full order model.
\end{abstract}


\section{Introduction}

Computational methods for real-time and many-query simulation of parametrized differential equations often require prohibitively high computational costs to achieve sufficiently accurate numerical solutions. During the last decades, model order reduction \cite{prudhomme02,BGW15,QMN16,HPRo22} has proved successful in providing low-complexity high-fidelity surrogate models that allow rapid and accurate simulations under parameter variation, thus enabling the numerical simulation of increasingly complex problems.
However, in the presence of operators with general nonlinear dependence
on the state, the computational cost of solving these surrogate models might still depend on the size of the underlying full model, resulting in simulation times that hardly improve over the original system simulation. While a non-intrusive, purely data-driven approach to model order reduction could in principle circumvent this bottleneck, we focus in this work on intrusive techniques, where numerical simulations of the surrogate model require potentially expensive evaluations of the nonlinear operators. This is a well-known issue in model order reduction and has led to the so-called hyper-reduction \cite{R09} methods.
Most of these techniques consist in approximating the high-dimensional nonlinear terms using sparse sampling via interpolation among samples of the nonlinear operators. This is the rationale behind missing point estimation \cite{AWWB08}, the empirical interpolation method (EIM) \cite{BMNP04,GMNP07}, the discrete empirical interpolation method (DEIM) \cite{CS10}, Gauss-Newton with approximated
tensors (GNAT) \cite{GNAT13} and the trajectory piecewise linear (TPWL) method \cite{tpwl}.
The computational efficiency of hyper-reduction techniques is based on the assumption that the nonlinear operator depends sparsely on the system state or, in other words, the approximation requires only the evaluation of few entries of the nonlinear vector field which, in turns, depend only on few entries of the reconstructed state. This assumption allows to reconstruct only few components of the high-dimensional state, thus reducing the cost of evaluating nonlinear terms to something proportional to the number of interpolation indices rather than the full dimension.
Although this assumption is satisfied whenever the nonlinear system at hand stems from a local discretization of a PDE, e.g., via finite element or finite volume schemes, many cases of interest are ruled out. For example, interacting particle systems often involve nonlinear operators that depend on the distance or on some interaction of each particle with all other particles in the system.
In such situations, hyper-reduction techniques can prove ineffective.
In this work we focus on one of such models, namely the system resulting from a 
particle-based discretization of the Vlasov-Poisson equation.

\textbf{Model order reduction of kinetic plasma models.}
The Vlasov-Poisson equation is a kinetic plasma model that describes the evolution of the distribution function of a family of collisionless charged particles moving under the action of a self-consistent electric field. 
Among various approximation techniques, particle-in-cell (PIC) discretizations \cite{BL91} approximate the distribution function of the plasma using a finite number of computational macro-particles, which are then advanced along the characteristics of the Vlasov equation. One of the major advantages of PIC discretizations is that they are able to preserve the Hamiltonian structure of the continuous Vlasov-Poisson model. However, 
the necessity of resolving the smallest length scales and the slow convergence of PIC methods entail that a large number of macro-particles is required to achieve accurate approximations. Therefore, the computational cost associated with the numerical simulation of the Vlasov-Poisson equation can be demanding even for one-dimensional problems, especially in multi-query scenarios, where a numerical solution has to be computed for many instances of input parameters.

The topic of model order reduction of the Vlasov-Poisson equation has received considerable attention in recent years. We focus here on low-rank approximation methods \cite{Kor15,KS16,EL17,TK23,TCGLCB23}, although we mention that alternative approaches, based on high-performance computing algorithms \cite{Gr06,RHK24} and 
neural network techniques \cite{CGZ24,BBP24,FDNV25}, have been developed for kinetic plasma models. Model order reduction of the Vlasov-Poisson equation is challenged by several factors. First, due to its multi-scale nature, the Vlasov-Poisson problem lacks favorable global reducibility properties, in the sense that large approximation spaces are typically required to achieve accurate approximations.
Second, since the Vlasov-Poisson equation admits a Hamiltonian formulation \cite[Section 1.6]{MR99}, it is crucial that numerical methods applied to this problem are designed so as to preserve its geometric structure at the discrete level.
Third, even if the dimension of the reduced order model is much smaller than that of the original problem, there is no hope of attaining significant reductions of the computational cost without an efficient treatment of the nonlinear operators involved. Indeed, the numerical solution of the Vlasov-Poisson problem via particle methods requires to determine the position of each macro-particle with respect to the computational grid employed for spatial discretization of the Poisson equation, the so-called particle-to-grid map. 
Since the knowledge of the expansion coefficients of the reduced order solution alone is not sufficient to determine the particles' positions, the approximated solution has to be reconstructed at each time step and for each test parameter, resulting in unbearable computational costs. 
To the best of our knowledge efficient hyper-reduction in this context is an open problem.

A recent work \cite{HPR24} addresses structure-preserving hyper-reduction of the Vlasov-Poisson equation in the number of particles by first approximating the electric potential via dynamic mode decomposition (DMD) \cite{Sch10} and then performing empirical interpolation of the particle-to-grid map. However, this approach relies on the fact that the approximation of the electric potential provided by DMD is sufficiently accurate, which might not be valid for a general range of parameters and over long times.

\textbf{Goals and contributions.}
In this work we consider a different approach and propose an adaptive structure-preserving hyper-reduction scheme. The effectiveness of the proposed method relies on the existence of a decomposition of the discrete Hamiltonian into a linear combination of terms, each depending on few components of the state. While this assumption is always satisfied in the case of local discretizations, such as finite differences or finite elements methods, this approach is not directly applicable to the Vlasov-Poisson problem or, in general, to Hamiltonian systems arising from PIC discretizations, as such a decomposition is not immediately available. The goal of this work is to generalize the setting of \cite{PV23,PV25} to account for nonlocal discretizations. The resulting hyper-reduction strategy is combined with an explicit time integrator for the evolution of the reduced basis and coefficients that allows to fully exploit the separability of the Vlasov-Poisson Hamiltonian. This results in a method that is completely decoupled from the full order problem, in the sense that it does not rely on any knowledge of the full order solution. Unlike classical non-intrusive approaches, however, the solution of a surrogate model is required, and the numerical complexity of this operation still depends on the full order dimension and on the number of test parameters. To address this, we propose a parameter sampling algorithm so that the arithmetic complexity of the resulting hyper-reduced system is linear in the full order dimension and in the number of test parameters, but does not depend on their product. A rank-adaptive strategy is developed to deal with changes in the reducibility of the solution set over time. Numerical experiments show that the hyper-reduced system provides stable and accurate simulations, while considerably reducing the runtime of the full order problem.

The remainder of the paper is organized as follows. 
\Cref{sec:non-sparse} is devoted to the illustration of the main limitations of hyper-reduction of non-sparse operators in the context of particle-based models.
In \Cref{sec:FOM} we recall the Vlasov-Poisson equations and its Hamiltonian formulation, and we define the full order model stemming from its PIC discretization. In \Cref{sec:ROM} the reduced order model is constructed and an explicit timestepping scheme is proposed for its temporal integration. The rank-adaptive approximation is described in \Cref{sec:rank-adaptive}. In \Cref{sec:hyper-reduction} we introduce a structure-preserving hyper-reduction method for the approximation of the Hamiltonian gradient, together with an \emph{a priori} convergence estimate. Finally, numerical experiments are reported in \Cref{sec:num_exp}.

\textbf{Notation.} Throughout the paper we will use capital letters to denote matrices and matrix-valued quantities, while we will use lower-case bold letters to denote vectors.
Given a matrix $A\in\Rbb^{n\times m}$, the element of $A$ in the $i$th row and $j$th column is denoted by $A_i^j\in\Rbb$, while $\Abf^j\in\Rbb^{n}$
denotes the $j$th column vector of $A$. Moreover, we denote by $\norm{A}_2$ and $\normF{A}$ the $2$-norm and the Frobenius norm of $A$, respectively.
Given a vector $\vbf\in\Rbb^n$, its $i$th entry is denoted by $\mathrm{v}_i\in\Rbb$, while its Euclidean norm is $\norm{\vbf}$.

\section{Hyper-reduction of non-sparse operators}\label{sec:non-sparse}

A major bottleneck in the construction of hyper-reduced models for particle systems is associated with the presence of nonlinear operators given as the sum of functions that depend nonlinearly on all particles of the system. Let $N$ denote the number of particles of the system and consider a function of the form
\begin{equation}\label{eq:hform}
    h(\xbf) = \sum_{i=1}^{\kappa}\sum_{\ell=1}^N f_i(\xbf_{\ell})
\end{equation}
where $\xbf\in\mathbb{R}^N$ is a vector associated with the degrees of freedom of the problem, e.g., the particles' positions, $f_i:\mathbb{R}\rightarrow\mathbb{R}$, and $\kappa, N\in\mathbb{N}$ with $\kappa\leq N$.
A typical example is when $f$ corresponds to the distance between two bodies in an interacting system, and it is thus given by $f_i(X_{\ell})=d(X_{\ell},X_i)$ where $X_i$ denotes the position of the $i$th body of the system, $d(\cdot,\cdot)$ denotes some distance, and, typically, $\kappa=N$.

We are interested in the hyper-reduction of the vector-valued function obtained by taking the gradient of $h$ in \eqref{eq:hform}, see \Cref{sec:hyper-reduction}.
Among the wide variety of hyper-reduction techniques, the empirical interpolation method (EIM) consists in a linear approximation in the span of $m\in\mathbb{N}$ basis vectors extracted from snapshots of the nonlinear operator to be reduced.
The EIM basis vectors are stored as columns of a matrix
$U\in\mathbb{R}^{N\times m}$ and $m$ pairwise distinct interpolation points are collected in the matrix $P\in\Rbb^{N\times m}$.
The EIM projection operator is then defined as $\mathbb{P}=U(P^{\top}U)^{-1}P^{\top}$ and the interpolant is
$$\mathbb{P}\nabla h(\xbf)=U(P^{\top}U)^{-1}P^{\top}\nabla h(\xbf),\qquad\forall\,\xbf\in\Rbb^{N}.$$
In principle, one could select $m+m^\prime$ sample points, with $m^\prime\geq1$ and define $P\in\Rbb^{N\times (m+m^\prime)}$. In this case, the interpolatory projection $\mathbb{P}$ is replaced with $U(P^{\top}U)^{\dagger}P^{\top}$, where $(P^\top U)^\dagger$ denotes the pseudoinverse of $P^\top U$. For the sake of simplicity we choose $m^\prime=0$ in this work, although our proposed method can also be applied to the case of oversampling.
The term $P^{\top} \nabla h(\xbf)$ samples the nonlinear function at $m$ components only, and if each of these components depend on few entries of $\xbf$, say $s\leq N$, then the computational cost to evaluate the hyper-reduced operator at each $\xbf\in\Rbb^{N}$ scales as $ms$ and does not depend on $N$. However, this assumption is not satisfied in many interesting cases.
Taking the gradient of a function as in \eqref{eq:hform} 
and performing hyper-reduction has two major limitations: (i) the gradient of $h$ results in a vector-valued function where each entry depends on all entries of the vector $\xbf$, making the evaluation of  $P^{\top} \nabla h(\xbf)$ computationally inefficient;
and (ii) the gradient structure of the nonlinear operator is not preserved.

\subsection{Gradient-preserving hyper-reduction of non-sparse operators}
To overcome the aforementioned limitations, we propose to re-write the nonlinear function $h$ in \eqref{eq:hform} as follows
\begin{equation}\label{eq:dec}
    h(\xbf)=\sum_{i=1}^\kappa \cbf_i \cdot F_i(\xbf),
    \qquad\forall\,\xbf\in\Rbb^{N},
\end{equation}
where, for any $1\leq i\leq \kappa$, $F_i:\mathbb{R}^N\rightarrow\mathbb{R}^N$ is defined such that the $\ell$th component of $F_i(\xbf)$, for $1\leq \ell\leq N$, is given by $F_i^\ell(\xbf)=f_i(\xbf_\ell)$. 
Here $\cbf_i\in\mathbb{R}^N$ is the vector with all entries equal to 1, for any $1\leq i\leq \kappa$, but more general situations can be accommodated by the decomposition \eqref{eq:dec}.
The idea is then to approximate not the gradient of $h$ but the function $h$ itself, written as in \eqref{eq:dec}, as
\begin{equation}\label{eq:hyp}
    h(\xbf)\approx \sum_{i=1}^\kappa \cbf_i \cdot \mathbb{P}_i F_i(\xbf),
    \qquad\forall\,\xbf\in\Rbb^{N},
\end{equation}
where $\{\mathbb{P}_i\}_{i=1}^{\kappa}$ are suitable projections associated with the chosen hyper-reduction technique.

Since the gradient operator is applied after hyper-reduction, the approximate operator is, by construction, still a gradient. Moreover,
for fixed $1\leq i\leq \kappa$, the vector-valued function $F_i$ depends sparsely on the data, in the sense that the $\ell$th entry of $F_i$ only depends on the $\ell$th entry of the variable.
This allows to efficiently apply hyper-reduction because evaluating $m\ll N$ entries of $F_i$ corresponds to a computational cost proportional to $m$ and not to $N$.

Although we believe that the decomposition proposed in \eqref{eq:dec} is sufficiently general and applicable to many problems, the specific choice of the functions $F_i$, of the coefficients $\cbf_i$, and of the projections $\mathbb{P}_i$ is problem-dependent and, even for one given problem, is not unique.
Therefore the optimal hyper-reduction strategy of the type \eqref{eq:hyp} will differ from problem to problem.
One aspect to consider is that the nonlinear operator needs to be reducible. As it turns out in many interacting particle systems, not all interactions are ``relevant'' to describe the dynamics of the system. Moreover, previous works have shown that dealing with nonlinear operators in the reduced space improves their reducibility; more precisely, the singular values of the snapshot matrix of \emph{projected} nonlinear operators typically exhibit a much faster decay, which ultimately results in a more efficient hyper-reduction, see \cite[Remark 3.2]{PV23}.

In this work we focus on particle-based discretizations of kinetic plasma models.
Here, the computational bottleneck is associated with the interaction between particle positions and electromagnetic fields via the so-called particle-to-grid map.

In previous works by the authors \cite{PV23,PV25}, a structure-preserving strategy based on a sparse decomposition of the Hamiltonian was proposed. The main requirement was that the reduced Hamiltonian could be written as a linear combination of $O(N)$ terms, each depending on a small number of entries of the state vector.
In the context of particle-in-cell discretizations, a sparse decomposition might not be directly available. One of the goals of this work, \Cref{sec:hyper-reduction} in particular, is to extend the framework of \cite{PV23,PV25} to address this shortcoming, and produce effective hyper-reduction techniques that retain the Hamiltonian structure of the problem.


\section{The Vlasov-Poisson equation and its numerical discretization}\label{sec:FOM}
In kinetic plasma models, the plasma is described in terms of a distribution function $f(t,x,v)$, representing the probability of having a particle occupying the position $x\in\Omega_x$ with velocity $v\in\Omega_v$ at time $t\in\Tcal=[0,T]$. 
Here we assume that $f$ also depends on a parameter $\prm$ belonging to a parameter space $\prms\subset\Rbb^P$, with $P\geq1$. In this work we focus on the 1D-1V problem by taking $\Omega:=\Omega_x\times\Omega_v\subset\mathbb{R}^2$ with $\Omega_x=[0,\ell_x]$ and $\Omega_v=\Rbb$.
The 1D-1V Vlasov-Poisson problem for one particle species and initial condition
$f(0,x,v;\prm) = f_0(\prm)$ reads
\begin{equation}\label{eq:VP-1D1V}
    \begin{cases}
        \partial_tf(t,x,v;\prm) + v\partial_xf(t,x,v;\prm)+\dfrac{q}{m}e(t,x;\prm)\partial_vf(t,x,v;\prm)=0,\\
        \partial_xe(t,x;\prm)=-\partial_{xx}\phi(t,x;\prm)
        =\rho_0+\rho(t,x;\prm)
        = \rho_0+q\displaystyle\int_{\Omega_v}f(t,x,v;\prm)\,dv,
    \end{cases}
\end{equation}
where $q$ is the electric charge, $m$ is the particle mass, 
$\rho$ is the electric charge density
and $\rho_0$ the electric charge density associated to a background charge $q_0$.
In the following, we assume that the problem is normalized so that $m=1$, $q=-1$ and $\rho_0=1$.
The function $e$ in \eqref{eq:VP-1D1V} is the unknown electric field, while $\phi$ is the associated electric potential, defined by $e(t,x;\prm)=-\partial_x\phi(t,x;\prm)$. The boundary conditions for $f$ are assumed to be periodic in space 
and we assume a Gaussian decay in velocity, namely
$f(t,x,v;\prm)\approx e^{-v^2}$ as $\lvert v\rvert\to\infty$.
Problem \eqref{eq:VP-1D1V} admits a Hamiltonian formulation with Hamiltonian given by the total energy \cite[Section 1.6]{MR99}
\begin{equation}\label{eq:Hamiltonian_cont}
    \Hcal(f)=\dfrac12\int_\Omega v^2f(t,x,v;\prm)\,dx\,dv+\dfrac{1}{2}\int_{\Omega_x}\lvert \partial_x\phi(t,x;\prm)\rvert^2\,dx.
\end{equation}

\subsection{Semi-discrete approximation}\label{sec:semidiscr}
One of the most used numerical discretizations of the Vlasov equation is
based on approximating the distribution function as the superposition of macro-particles.
We consider a particle method for the approximation of the Vlasov equation, coupled with a finite element discretization of the Poisson problem for the electric potential $\phi$. This choice results in a semi-discrete Hamiltonian system \cite{HSQL16,KKMS17}. More in detail, the distribution function $f$ is approximated by
\begin{equation}\label{eq:fpic}
    f_h(t,x,v;\prm)=\sum_{\ell=1}^N\omega_\ell\delta(x-\pos_\ell(t,\prm))\delta(v-\vel_\ell(t,\prm)),
\end{equation}
where $\delta$ is the Dirac delta and $\pos_\ell$ and $\vel_\ell$ denote the position and velocity of the $\ell$th macro-particle, respectively.
The weights $\{\omega_\ell\}_{\ell=1}^N$ in the expansion \eqref{eq:fpic} are assumed to be all equal, i.e. $\omega_\ell=\omega$ for all $\ell=1,\dots,N$ and $\omega$ is determined by integrating the Poisson equation \eqref{eq:VP-1D1V} in space and enforcing the periodic boundary conditions. This yields
$\omega=\ell_x N^{-1}$.

The time evolution of $f_h$ is derived by advancing the macro-particles along the characteristics of the Vlasov equation, which gives
\begin{equation*}
\begin{cases}
    \dpos(t,\prm)=\velbf(t,\prm), \\
    \dvel(t,\prm)=-e(t,\posbf(t,\prm);\prm),
\end{cases}
\end{equation*}
where $\posbf(t,\prm)\in\Rbb^N$ and $\velbf(t,\prm)\in\Rbb^N$
are the vector-valued quantities collecting the particles positions and velocities, respectively, at time $t$ and for a fixed parameter $\prm$.

To compute the electric field $e$ we approximate the Poisson equation with a finite element discretization. Let us consider a uniform partition of the spatial interval $\Omega_x$ as $0=x_0<x_1<\dots<x_{N_x}=\ell_x$, with $x_i=i\Delta x$ for $i=0,\dots,N_x$ and $\Delta x=\ell_x/N_x$, where $x_0$ is identified with $x_{N_x}$ in view of periodic boundary conditions in space. Let $\mathcal{P}_\pdeg(\Omega_x)\subset H^1(\Omega_x)$ be the space of continuous piecewise polynomials of degree at most $\pdeg\geq 1$ on $\Omega_x$, subject to periodic boundary conditions. Note that $\mathcal{P}_\pdeg(\Omega_x)$ is a linear subspace of $H^1(\Omega_x)$ of dimension $\pdeg N_x$. The variational problem associated to the Poisson equation reads: for every $t\in\Tcal$ and $\prm\in\prms$, find $\phi_h=\phi_h(t,\cdot;\prm)\in\mathcal{P}_\pdeg(\Omega_x)$ such that
$a(\phi_h,\psi)=g_h(\psi)$, for all $\psi\in\mathcal{P}_\pdeg(\Omega_x)$, where the bilinear form $a$ and the linear operator $g_h$ are defined as
\begin{equation*}
    a(\phi,\psi) := \int_{\Omega_x}\phi^\prime(x)\psi^\prime(x)\,dx \quad \text{ and } \quad g_h(\psi):=\int_{\Omega_x}\psi(x)\,dx-\int_\Omega f_h(t,x,v;\prm)\psi(x)\,dx\,dv.
\end{equation*}
Let $\fedim:=\pdeg N_x-1$ and $\{\lambda_i(x)\}_{i=1}^{\fedim+1}$ be a basis of $\mathcal{P}_\pdeg(\Omega_x)$. 
Let $\boldsymbol{\Phi}(t,\prm)\in\Rbb^{\fedim}$ be the vector of expansion coefficients of the semi-discrete potential $\phi_h\in\mathcal{P}_\pdeg(\Omega_x)$ in the basis $\{\lambda_i(x)\}_{i=1}^{\fedim}$,
where the $(\fedim+1)$th coefficient has been set to $0$ to single out a solution of the Poisson problem.
The variational problem associated with the Poisson equation can then be written as follows: for any $t\in\Tcal$ and $\prm\in\prms$, find
$\boldsymbol{\Phi}(t,\prm)\in\Rbb^{\fedim}$ such that
\begin{equation*}
    \stmat\boldsymbol{\Phi}(t,\prm) =  \decd(\posbf(t,\prm)),
\end{equation*}
where $\stmat\in\Rbb^{\fedim\times\fedim}$
is the stiffness matrix
defined as $\stmat_{i}^{j}:=a(\lambda_j,\lambda_i)$
and $\decd(\posbf)\in\Rbb^{\fedim}$ is the discrete electric charge density, whose $j$th entry, for $j=1,\dots,\fedim$, can be computed
using the approximation \eqref{eq:fpic} of $f$, as
\begin{equation}\label{eq:gj}
    \decdi_j(\posbf) = g_h(\lambda_j)= \int_{\Omega_x}\lambda_j(x)\,dx-\frac{\ell_x}{N}\sum_{\ell=1}^N\lambda_j(\pos_\ell(t,\prm)).
\end{equation}
By introducing the vector $\sbf\in\Rbb^\fedim$ whose $j$th entry is given by $\sbfi_j=\int_{\Omega_x}\lambda_j(x)\,dx$, and the matrix-valued function $\Lambda$ defined as 
\begin{equation*}
    \xbf\in\Rbb^N \mapsto \Lambda(\xbf)\in\Rbb^{N\times\fedim} \quad \text{ such that } \quad \Lambda(\xbf)_{\ell}^{ i}:=\lambda_i(\pos_\ell),
\end{equation*}
the discrete electric charge density can be written as
$\decd(\posbf(t,\prm))=\sbf-\ell_x N^{-1}\Lambda(\posbf(t,\prm))^\top\mathbf{1}_N$,
where $\mathbf{1}_N\in\Rbb^N$ denotes the vector with all entries equal to $1$.

Then, the discrete Hamiltonian resulting from the discretization of \eqref{eq:Hamiltonian_cont} reads
\begin{equation*}
    \Hcal(f_h)=
    \dfrac{\ell_x}{2N}\sum_{\ell=1}^N\vel_\ell^2 + \dfrac12\mathbf{\Phi}(\posbf)^\top\stmat\mathbf{\Phi}(\posbf).
\end{equation*}
Using the discretization of the Poisson equation and re-normalizing
with respect to the multiplicative constant $\ell_x N^{-1}$, the discrete Hamiltonian can be defined as a function of the unknown particles' positions and velocities as
\begin{equation}\label{eq:Hamiltonian}
    \ham(\posbf,\velbf)
    =\dfrac12\velbf^\top \velbf+h(\posbf)
    =\dfrac12\velbf^\top \velbf+\dfrac{N}{2\ell_x}\decd(\posbf)^\top \stmat^{-1}\decd(\posbf),
\end{equation}
where $\nlham$ is the discrete electric potential energy
and corresponds to the non-quadratic part of the Hamiltonian.
Notice that the Hamiltonian associated with the Vlasov-Poisson problem is separable, that is, the kinetic energy only depends on the particles' velocities, while $\nlham$ only depends on the particles' positions.
Let us introduce the matrix-valued function $\nabla\Lambda:\xbf\in\Rbb^N\mapsto \nabla\Lambda(\xbf)\in\Rbb^{N\times\kappa}$, defined as $\nabla\Lambda(\xbf)_{\ell}^{j}:=\lambda_j^\prime(\pos_{\ell})$. The gradient of the nonlinear part of the Hamiltonian with respect to the particles' positions is associated with the discrete electric field, and it is given by
\begin{equation}\label{eq:nonlin_ham_FOM}
    \nabla_{\posbf}\ham(\posbf,\velbf) = \nabla_{\posbf}h(\posbf)=-\nabla\Lambda(\posbf)\stmat^{-1}\decd(\posbf).
\end{equation}
Finally, the semi-discrete system in Hamiltonian form reads
\begin{equation}\label{eq:semidisc}
    \begin{bmatrix}
    \dpos(t,\prm) \\
    \dvel(t,\prm)
    \end{bmatrix}=J_{2N}
    \begin{bmatrix}
    -\nabla\Lambda(\posbf(t,\prm))\stmat^{-1}\decd(\posbf(t,\prm)) \\
    \velbf(t,\prm)\end{bmatrix}
    =\begin{bmatrix}
    \velbf(t,\prm) \\
    -\nabla_{\posbf}\nlham(\posbf(t,\prm))
    \end{bmatrix},
\end{equation}
where $J_{2N}\in\Rbb^{2N\times 2N}$ is the canonical symplectic tensor defined as
\begin{equation*}
    J_{2N}=\begin{bmatrix}
        0_N & I_N \\ -I_N & 0_N
    \end{bmatrix}\in\Rbb^{2N\times 2N},
\end{equation*}
with $I_N$, $0_N\in\Rbb^{N\times N}$ denoting the identity and the zero matrix of dimension $N$, respectively.

Suppose we are interested in solving the semi-discrete problem \eqref{eq:semidisc} for $p$ test parameters $\{\prm_1,\dots,\prm_p\}\in\prms$. This is a case of interest in multi-query contexts such as uncertainty quantification or optimal experimental design. To this end we introduce the matrix-valued function $\fX(t)\in\Rbb^{N\times p}$, whose $(\ell,s)$-entry denotes the position of the $\ell$th particle associated to the $s$th parameter, that is
$\fX_{\ell}^s(t)=\pos_{\ell}(t,\prm_s)$
for any $1\leq\ell\leq N$ and $1\leq s\leq p$.
Similarly, for the velocity variables, $\fV(t)\in\Rbb^{N\times p}$ is defined as
$\fV_{\ell}^s(t)=\vel_{\ell}(t,\prm_s)$.
Problem \eqref{eq:semidisc} then becomes: given $\fos(0)\in\Rbb^{2N\times p}$ find $\fos(t)\in\Rbb^{2N\times p}$ such that
\begin{equation}\label{eq:FOM}
    \dfos(t) := \begin{bmatrix}
        \dfX(t) \\ \dfV(t)
    \end{bmatrix} = \begin{bmatrix}
        \fV(t) \\ -\mdiscref(\fX(t))
    \end{bmatrix}
\end{equation}
where $\mdiscref(\fX(t))\in\Rbb^{N\times p}$ is defined as
$\mdiscref(\fX(t))_{\ell}^s
=\partial_{\pos_{\ell}}h(\posbf(t,\prm_s))$
for any $1\leq\ell\leq N$ and $1\leq s\leq p$.
In the following, we refer to \eqref{eq:FOM} as the full order model (FOM).

\subsection{Numerical time integration}\label{sec:FOMtime}

In this work, the full order model \eqref{eq:FOM} is solved in time by means of the St\"ormer-Verlet scheme. In addition to being a symplectic integrator \cite[Theorem 3.4]{HLW10}, the St\"ormer-Verlet scheme is explicit for separable Hamiltonian systems. This choice allows to circumvent the need for a nonlinear solver at each time step, which might lead to a considerable computational effort in the presence of large-scale systems to be simulated for many parameters. Introducing the uniform grid in time $0=t^{0}<t^{1}<\dots<t^{\nt}=T$, with $t^{\tsind}=\tsind\Delta t$ for $\tsind=0,\dots,\nt$ and $\Delta t = T/\nt$, and defining $\fX^{(\tsind)}$ and $\fV^{(\tsind)}$ as approximations of $\fX(t^\tsind)$ and $\fV(t^\tsind)$, respectively, the numerical time integration of \eqref{eq:FOM} in the temporal subinterval $(t^{\tsind},t^{\tsind+1}]$, for any $\tsind=1,\dots,\nt$, reads
\begin{align}
\begin{split}\label{eq:SV-full}
    &\fX^{(\tsind)}=\fX^{(\tsind-1)}+\Delta t\left(\fV^{(\tsind-1)}-\dfrac{\Delta t}{2}\mdiscref(\fX^{(\tsind-1)})\right), \\
    & \fV^{(\tsind)} = \fV^{(\tsind-1)}-\dfrac{\Delta t}{2}\bigg(\mdiscref(\fX^{(\tsind-1)})+\mdiscref(\fX^{(\tsind)})\bigg),
\end{split}
\end{align}
with given initial conditions $\fX^{(0)}$ and $\fV^{(0)}$.

The computational cost of solving \eqref{eq:SV-full} for one time step is $O(Np\pdeg)$. While the update of the particles' positions and velocities requires $O(Np)$ operations, the bulk of the computational cost is required by the evaluation of the discrete electric field $E(\fX)$.
Indeed, for each test parameter $\prm_s$, with $1\leq s\leq p$, one has to evaluate the quantity
$-\nabla\Lambda(\fXbf^s)\stmat^{-1}\decd(\fXbf^s)$.
This requires to evaluate the particle-to-grid map by associating to each macro-particle the mesh element containing it, with complexity $O(N)$ for each instance of the parameter. Then, assuming that the number of operations required to evaluate the basis functions and their derivatives is constant, the matrices $\Lambda(\fXbf^s)$ and $\nabla\Lambda(\fXbf^s)$ are assembled in $O(N\pdeg)$ operations: in fact, every particle is contained in the support of at most $\pdeg+1$ basis functions. 
A similar cost is needed for the computation of the discrete electric charge density $\decd(\fXbf^s)$.

\section{Symplectic dynamical low-rank approximation}\label{sec:ROM}

Due to their slow convergence rate, particle-based methods require a large number of computational macro-particles to achieve accurate approximations. This entails that the computational cost of solving \eqref{eq:SV-full} can become prohibitive, especially in the presence of a large number $p$ of test parameters. To alleviate this computational burden, projection-based model order reduction \cite{BGW15} aims at representing the full order dynamics in a low-dimensional approximation space.

Given $n\in\mathbb{N}$, with $n\leq N$ and $n\leq p$, the full order solution $\fos(t)\in\Rbb^{2N\times p}$ is approximated by $A\coef(t)$ where the columns of the reduced basis matrix $A\in\Rbb^{2N\times 2n}$ span the $2n$-dimensional, so-called, reduced basis space and $\coef\in\Rbb^{2n\times p}$ contains the time-dependent expansion coefficients for all test parameters.
In order to preserve the Hamiltonian structure of the original problem, the reduced basis matrix $A$ is required to be orthosymplectic, that is, to satisfy $A^\top A=I_{2n}$ and $A^\top J_{2N}A=J_{2n}$ \cite{PM16}. For the sake of computational efficiency, it is desirable that the dimension of the reduced space is much smaller than the full order dimension $N$. On the other hand, this is only possible if the Kolmogorov $n$-width of the solution set associated to the FOM \cite{Pin85} decays sufficiently fast with $n$.
In this scenario, a small approximation space is sufficient to yield an accurate representation of the full order dynamics for all parameters and at all times. The conservative nature of Hamiltonian system, however, results in unfavorable \emph{global} reducibility properties, meaning that large reduced spaces are required to achieve even moderate accuracy. If the problem is \emph{locally} reducible, this issue can be addressed by evolving the reduced space over time.
This has led to the development of nonlinear model order reduction techniques where the reduced space, and, possibly, its dimension, evolve over time, see e.g. \cite{EL19,EiL18,EJ21,UZ24} in the context of kinetic plasma models.
In this work we follow the approach proposed in \cite{Pag21}, and consider an approximation of the full order solution $\fos(t)\in\Rbb^{2N\times p}$ of the form $A(t)\coef(t)$ where the reduced basis $A$ is time-dependent.

\textbf{The local reduced basis space.}
To satisfy the orthosymplecticity constraint, the reduced basis matrix $A(t)$ must be of the form
\begin{equation*}
    A(t) = \begin{bmatrix}
        \redb(t) & -\widehat{\redb}(t) \\ \widehat{\redb}(t) & \redb(t)
    \end{bmatrix}
\end{equation*}
with $\redb(t),\widehat{\redb}(t)\in\Rbb^{N\times n}$ satisfying $\redb^\top(t)\redb(t)+\widehat{\redb}(t)^\top\widehat{\redb}(t) = I_{n}$ and $\redb^\top(t)\widehat{\redb}(t)+\widehat{\redb}^\top(t)\redb(t) = 0_n$ at each time \cite[Lemma 4.4]{PM16}. Moreover, if $\widehat{\redb}\neq0$, the reduced order model does not preserve the separability of the full order Hamiltonian. This precludes the possibility of employing an explicit time integrator for the reduced system, which might lead to suboptimal performances in terms of computational time. For this reason, in this work, we impose that $A$ possesses a block-diagonal structure by setting $\widehat{\redb}(t)=0$. More precisely, given an orthogonal matrix $\redb(t)\in\Rbb^{N\times n}$ at a fixed time $t\in\Tcal$, we approximate the particles' positions $\fX(t)$ and velocities $\fV(t)$ with $\rX(t)=\redb(t)\cX(t)$ and $\rV(t)=\redb(t)\cV(t)$, respectively, where $\cX(t),\cV(t)\in\Rbb^{n\times p}$. 
This can be seen as approximating $\fos(t)$ at any given time $t$ in the set
\begin{equation*}
    \mathbb{S}_n:=\left\{\ros=\begin{bmatrix}\rX \\ \rV\end{bmatrix}\in\Rbb^{2N\times p} : \rX=\redb\cX,\, \rV=\redb\cV \text{ with } \redb\in\stm \text{ and } \text{rank}(\cX\cX^\top + \cV\cV^\top)=n\right\}
\end{equation*}
where $\stm=\{\redb\in\Rbb^{N\times n}: \redb^\top\redb = I_n\}$ is the Stiefel manifold. It can be shown that the rank condition $\text{rank}(\cX\cX^\top + \cV\cV^\top)=n$ on the coefficients
endows $\mathbb{S}_n$ with a manifold structure.

\textbf{The reduced order model.}
The coefficient matrices $\cX,\cV\in\Rbb^{n\times p}$ have columns defined as
$\cXbf^s(t) := \rposbf(t,\prm_s)$ and $\cVbf^s(t) := \rvelbf(t,\prm_s)$
for any $1\leq s\leq p$.
Similarly to \cite{Pag21}, following a dynamical low-rank approximation approach \cite{KL07}, evolution equations for the reduced basis and the coefficients can be obtained by projecting \eqref{eq:FOM} onto the tangent space at $\ros$ to the approximation manifold $\mathbb{S}_n$ via a symplectic projection, as described in details in \cite[Section 4.1]{Pag21}.
This results in the following evolution equations for the basis $\redb$ and for the coefficients $\cX$ and $\cV$ associated with the particles' positions and velocities, respectively:
\begin{subequations}\label{eq:rom}
\begin{empheq}[left=\empheqlbrace]{align}
   & \dcX(t) = \cV(t) \label{eq:ZXdot} \\
   & \dcV(t) = -\redb(t)^\top\mdiscref\big(\redb(t)\cX(t)\big) \label{eq:ZVdot}
   \\ 
    & \Dot{\redb}(t)=\big(\redb(t)\redb(t)^\top-I_N\big)\mdiscref\big(\redb(t) \cX(t)\big)\cV(t)^\top M^{-1}\big(\cX(t),\cV(t)\big),
   \label{eq:Phidot}
\end{empheq}
\end{subequations}
where $M(\cX,\cV):=\cX\cX^\top + \cV\cV^\top$.
We will refer to the approximate system above as the reduced order model (ROM). Note that the approximate dynamics for the coefficients \eqref{eq:ZXdot}-\eqref{eq:ZVdot} is still Hamiltonian, with the reduced Hamiltonian $\rham$ given by
\begin{equation}\label{eq:red_Ham}
    \rham(\rposbf,\rvelbf;\redb) := \ham(\redb\rposbf,\redb\rvelbf)
    = \dfrac12\rvelbf^\top \rvelbf + \nlham(\redb\rposbf) \qquad \forall \rposbf,\rvelbf\in\Rbb^n.
\end{equation}
Note that the separability of the Hamiltonian is preserved. 
Moreover, the nonlinear part of the reduced Hamiltonian reads, for any $\rposbf\in\Rbb^n$,
\begin{equation}\label{eq:nonlinear_Ham}
    \nlrham(\rposbf;\redb) = \nlham(\redb\rposbf)
    = \dfrac{N}{2\ell_x}\decd(\redb\rposbf)^\top\stmat^{-1}\decd(\redb\rposbf)
    =\dfrac{N}{2\ell_x}\sum_{i,j=1}^\fedim\decdi_i(\redb\rposbf)^\top\stmat^{-1}_{ij}\decdi_j(\redb \rposbf)
\end{equation}
where $\decdi_j$, for any $1\leq j\leq \kappa$, is defined as in \eqref{eq:gj}.
The nonlinear Hamiltonian corresponds to the discrete potential energy associated with the approximate solution. 



\subsection{Partitioned Runge-Kutta scheme for the reduced dynamics}
\label{sec:pRK}
For the numerical time integration of the reduced order model we couple a St\"ormer-Verlet time integration scheme for the Hamiltonian system \eqref{eq:ZXdot}-\eqref{eq:ZVdot}
with a tangent method \cite{CO02,Pag21} for the numerical integration of the evolution equation \eqref{eq:Phidot} for the reduced basis.
Using a tangent method for \eqref{eq:Phidot} ensures that the time approximation of the basis $\redb$ remains orthogonal and that the computational cost associated with its discrete evolution remains linear in $N$.

First we recall from \cite{Pag21} how to discretize the time evolution of the reduced basis $\Psi$ and then we formulate a numerical method for the solution of the coupled system \eqref{eq:rom}.
More in detail, given an approximation $\redb^{(\tsind)}\in\stm$ to $\redb(t^{\tsind})$, it is possible to represent the reduced basis matrix $\redb$ at the generic time $t$ in the subinterval $(t^{\tsind}, t^{\tsind+1}]$ as the image of some matrix $\prpsi$ in the tangent space $T_{\redb^{(\tsind)}}\stm$ through the local retraction $\retr_{\redb^{(\tsind)}}:T_{\redb^{(\tsind)}}\stm\to\stm$, defined as
\begin{equation}\label{eq:retraction}
    \retr_{\redb^{(\tsind)}}(\prpsi):=\mathrm{cay}\big(\zeta_{\redb^{(\tsind)}}(\prpsi)(\redb^{(\tsind)})^{\top}-\redb^{(\tsind)} \zeta_{\redb^{(\tsind)}}(\prpsi)^{\top}\big)\redb^{(\tsind)}
\end{equation}
where $\zeta_{\redb^{(\tsind)}}(\prpsi):=(I-\redb^{(\tsind)}(\redb^{(\tsind)})^\top/2)\prpsi$ and $\mathrm{cay}$ denotes the Cayley transform \cite[Section IV.8.3]{HLW10}. Using the retraction \eqref{eq:retraction} and the evolution equation for $\redb$ in \eqref{eq:Phidot}, it is possible to write the following equation on the tangent space $T_{\redb^{(\tsind)}}\stm$: given $\prpsi(t^\tsind)$, find $\prpsi(t)$ such that
\begin{equation*}
    \dot{\prpsi}(t) = \big({d\retr_{\redb^{(\tsind)}}}[\prpsi(t)]\big)^{-1}\mathcal{L}\big(\cX(t),\cV(t),\retr_{\redb^{(\tsind)}}(\prpsi(t))\big)M^{-1}\big(\cX(t),\cV(t)\big) \qquad \forall t\in(t^\tsind,t^{\tsind+1}],
\end{equation*}
where $\mathcal{L}(\cX,\cV,\redb):=(\redb\redb^\top-I_N)\mdiscref(\redb \cX)\cV^\top$ and
${d\retr_{\redb^{(\tsind)}}}[\prpsi]:T_{\redb^{(\tsind)}}\stm\to T_{\retr_{\redb^{(\tsind)}}(\prpsi)}\stm$ is the tangent map of the retraction \cite[Section 5.3.1]{Pag21}. This approach ensures that the orthogonality constraint on $\redb(t)$ is satisfied at each time.
Therefore, the reduced order model \eqref{eq:rom} yields the following system of evolution equations in the time interval $(t^{\tsind},t^{\tsind+1}]$: given $\cX^{(\tsind)},\cV^{(\tsind)}$ and $\redb^{(\tsind)}$, solve
\begin{equation}\label{eq:evoleq}
    \left\{\begin{aligned}
& \dcX = \cV =: f_1(\cV)\\
& \dcV = -\retr_{\redb^{(\tsind)}}(\prpsi)^\top \mdiscref(\retr_{\redb^{(\tsind)}}(\prpsi)\cX) =: f_2(\cX,\prpsi)\\
& \Dot{\prpsi} = \big({d\retr_{\redb^{(\tsind)}}}[\prpsi]\big)^{-1}\mathcal{L}\big(\cX,\cV,\retr_{\redb^{(\tsind)}}(\prpsi)\big)M^{-1}(\cX,\cV)=:f_3(\cX,\cV,\prpsi)
    \end{aligned},\right.
\end{equation}
with $\prpsi^{(\tsind)}=0$ since $\retr_{\redb^{(\tsind)}}(0) = \redb^{(\tsind)}$.
Note that the retraction and its inverse tangent map can be evaluated with $O(Nn^2)$ operations as shown in \cite[Section 5.3.1]{Pag21}.

\textbf{A second order partitioned Runge-Kutta (RK) scheme.}
We combine an explicit time integrator for the evolution of $\prpsi$ with a symplectic integrator for the coefficient matrices, so that the geometric structure of the phase space of the reduced system is preserved at the time-discrete level.
We describe the method for a generic autonomous system of coupled differential equations
\begin{equation*}
    \begin{cases}
        \Dot{x}_1(t)=f_1(x_1(t),x_2(t),x_3(t)) \\
        \Dot{x}_2(t)=f_2(x_1(t),x_2(t),x_3(t)) \\
        \Dot{x}_3(t)=f_3(x_1(t),x_2(t),x_3(t))
    \end{cases}
\end{equation*}
with initial conditions $x_l(0)=x_l^{(0)}$, for $l\in\{1,2,3\}$. Note that \eqref{eq:evoleq} is a particular case obtained by identifying $x_1$, $x_2$ and $x_3$ with $\cX$, $\cV$ and $\prpsi$, respectively.  Given a time step $\Delta t$, we consider the following $\ns$-stage partitioned RK scheme:
for any $l\in\{1,2,3\}$ the numerical solution at time $t^1=t^0+\Delta t$ is given by
\begin{align}\begin{split}\label{eq:PRK3}
    &k^{[l]}_{i} = f_{l}\Big(x_1^{(0)}+\Delta t\sum_{j=1}^{\ns}a^{[1]}_{ij}k^{[1]}_j, x_2^{(0)}+\Delta t\sum_{j=1}^{\ns} a^{[2]}_{ij}k^{[2]}_j, x_3^{(0)}+\Delta t\sum_{j=1}^{\ns} a^{[3]}_{ij}k^{[3]}_j\Big), \quad i=1,\dots,\ns, \\ 
    &x_l^{(1)}=x_l^{(0)}+\Delta t\sum_{i=1}^{\ns}b_i^{[l]}k^{[l]}_i.
    \end{split}
\end{align}
If, for all $l\in\{1,2,3\}$, $(\{a^{[l]}_{ij}\}_{i,j=1}^{\ns},\{b_i^{[l]}\}_{i=1}^{\ns})$ are the coefficients of RK methods of order $2$, then \eqref{eq:PRK3} is a partitioned RK method of order $2$ if and only if
\begin{equation*}
    \sum_{i=1}^{\ns}b_i^{[l]}\sum_{j=1}^{\ns} a^{[r]}_{ij}
    = \dfrac12,\quad\forall\,l,r\in\{1,2,3\},\,l\neq r.
\end{equation*}
We recall that the aforementioned RK methods are of order $2$ if the following conditions are satisfied (see e.g. \cite[Section II.1.1]{HLW10}):
\begin{equation*}
    \sum_{i=1}^{\ns}b^{[l]}_i = 1,
    \qquad \sum_{i=1}^{\ns}b^{[l]}_i\sum_{j=1}^{\ns} a^{[l]}_{ij} =\frac12,\qquad\forall\, l\in\{1,2,3\}.
\end{equation*}
One possibility to obtain a two-stage, second order \emph{explicit} partitioned RK method for \eqref{eq:evoleq} is to combine the St\"ormer-Verlet scheme with the Heun method. This corresponds to the following choice of coefficients:
\begin{align*}
    &a_{21}^{[1]}=a^{[1]}_{22}=\frac12, \quad a^{[1]}_{11}=a^{[1]}_{12}=0, \quad b^{[1]}_1=b^{[1]}_2=\frac12, \\
    &a^{[2]}_{11}=a^{[2]}_{21}=\frac12, \quad a^{[2]}_{12}=a^{[2]}_{22}=0, \quad b^{[2]}_1=b^{[2]}_2=\frac12, \\
    &a^{[3]}_{21}=1, \quad a^{[3]}_{11}=a^{[3]}_{12}=a^{[3]}_{22}=0, \quad b^{[3]}_1=b^{[3]}_2=\frac12.
\end{align*}

The arithmetic complexity of solving the reduced model \eqref{eq:evoleq} with the proposed time integrator is $O(Npn)+O(Np\pdeg)$.
The major computational bottleneck is once more the
evaluation of the electric field.
Although the quantity $\redb^{\top} E(\redb\cX)$ has dimension $n \times p$, we still need to reconstruct the approximate particles' positions $\rX = \redb\cX$ at a cost of $O(Npn)$, and then evaluate the electric field in the high-dimensional space at a cost of $O(Np\pdeg)$, as explained in \Cref{sec:FOMtime}.

The approach we have presented in this section differs substantially from classical model order reduction techniques in that it does not rely on the so-called offline/online decomposition. Unlike standard MOR methods where the reduced space is constructed in the offline phase from snapshots generated by simulations of the full order model, in our setting the reduced basis is initialized from the initial condition and it is then evolved over time. This makes our approach completely independent of the full order model, and, in particular, it circumvents the need for a potentially expensive offline phase. 
Moreover, relying on one reduced basis space, as in traditional MOR, is not an option here since the transport nature of the Vlasov-Poisson problem makes it simply not reducible by a linear approximation space.

Dynamically updating the reduced basis space has a computational complexity that inevitably depends on the full order dimension $N$, although only \emph{linearly}.
Moreover, since one reduced basis is constructed for all $p$ test parameters, the complexity of this evolution will also scale with $p$. Therefore, in general, the dependence on $N$ and $p$ of the overall computational cost cannot be overcome. This entails that, despite dimensionality reduction, solving the reduced Vlasov-Poisson model for one time step is as computationally expensive as solving the full order model \eqref{eq:SV-full}. A significant improvement of computational efficiency can be achieved by decoupling the operations that depend on $N$ from those that depend on $p$. This will be the subject of \Cref{sec:hyper-reduction}.

\section{Rank adaptivity}\label{sec:rank-adaptive}
The dimension of the reduced model plays an important role in the context of model order reduction and dynamical approximation. On the one hand, a too small reduced space might yield a poor approximation of the full order solution, on the other hand, a large value might spoil the computational efficiency of the method and lead to overapproximation \cite[Section 5.3]{KL07}. Indeed, a necessary (but not sufficient) condition for the matrix $\cX\cX^\top+\cV\cV^\top$ in \eqref{eq:Phidot} to be full rank is that $n\leq p$, and violating this condition results in a rank-deficient evolution problem for the reduced basis. In order to address this issue, we propose an algorithm where the dimension of the reduced space is adapted over time. To this end, we apply to the Vlasov-Poisson problem the error indicator proposed in \cite{HPR23}, based on the linearized residual of the full order model. This error indicator is then used to determine when and how to modify the rank of the approximated solution. While the computation of the error indicator of \cite{HPR23} requires to solve a linear system of size $2N\times2N$ in the general case, we show that the particular structure of the Vlasov-Poisson problem allows for a much cheaper implementation.

Let us introduce the discrete residual associated with the St\"ormer-Verlet time integration scheme, that is
\begin{equation}\label{eq:full_residual}
    \resid_{\fom}^{(\tsind)} = \resid(\fos^{(\tsind)},\fos^{(\tsind-1)}) := \begin{bmatrix}
        \fX^{(\tsind)}-\fX^{(\tsind-1)}-\Delta t\left(\fV^{(\tsind-1)}-\dfrac{\Delta t}{2}\mdiscref(\fX^{(\tsind-1)})\right) \\ \fV^{(\tsind)}-\fV^{(\tsind-1)}+\dfrac{\Delta t}{2}\left(\mdiscref(\fX^{(\tsind-1)})+\mdiscref(\fX^{(\tsind)})\right)
    \end{bmatrix}.
\end{equation}
Let us first assume that $p=1$. The Taylor expansion of the residual map \eqref{eq:full_residual} around $(\ros^{(\tsind)},\ros^{(\tsind-1)})$ truncated at the first order reads, for all $\tsind$,
\begin{equation*}
    \resid_{\fom}^{(\tsind)}\approx\resid_{\rom}^{(\tsind)}+\frac{\partial\resid_{\fom}^{(\tsind)}}{\partial\fos^{(\tsind)}}\bigg|_{(\ros^{(\tsind)},\ros^{(\tsind-1)})}(\fos^{(\tsind)}-\ros^{(\tsind)})+\frac{\partial\resid_{\fom}^{(\tsind)}}{\partial\fos^{(\tsind-1)}}\bigg|_{(\ros^{(\tsind)},\ros^{(\tsind-1)})}(\fos^{(\tsind-1)}-\ros^{(\tsind-1)}),
\end{equation*}
where $\resid_{\rom}^{(\tsind)}=\resid(\ros^{(\tsind)},\ros^{(\tsind-1)})$ and 
\begin{equation}\label{eq:inv}
    \left(\dfrac{\partial\resid_{\fom}^{(\tsind)}}{\partial\fos^{(\tsind)}}\bigg|_{(\ros^{(\tsind)},\ros^{(\tsind-1)})}\right)^{-1}
    = \begin{bmatrix}
        I_N & 0_N \\
        -\dfrac{\Delta t}{2}\jacE(\rX^{(\tsind)}) & I_N
    \end{bmatrix},
\end{equation}
\begin{equation}\label{eq:inv2}
    \dfrac{\partial\resid_{\fom}^{(\tsind)}}{\partial\fos^{(\tsind-1)}}\bigg|_{(\ros^{(\tsind)},\ros^{(\tsind-1)})}=\begin{bmatrix}
        -I_N + \dfrac{\Delta t^2}{2}\jacE(\rX^{(\tsind-1)}) & -\Delta tI_N \\ \dfrac{\Delta t}{2}\jacE(\rX^{(\tsind-1)}) & -I_N
    \end{bmatrix}.
\end{equation}
Here $\jacE$ is the Jacobian matrix of the map $\xbf\in\Rbb^{N}\mapsto \mdiscref(\xbf)$ and it can be derived from the expression of the gradient of the nonlinear Hamiltonian part \eqref{eq:nonlin_ham_FOM}, that is
\begin{equation}\label{eq:jacobian_E}
    \jacE(\xbf) = -\text{diag}(\nabla^2\Lambda(\xbf)\stmat^{-1}{\decd}(\xbf))+\ell_xN^{-1}\nabla\Lambda(\xbf)\stmat^{-1}\nabla\Lambda(\xbf)^\top\in\Rbb^{N\times N},
\end{equation}
where $\text{diag}(\abf)$ denotes the diagonal matrix with diagonal elements given by the vector $\abf\in\Rbb^{N}$,
and $\nabla^2\Lambda(\xbf)\in\Rbb^{N\times\kappa}$ is defined as $\nabla^2\Lambda(\xbf)_{\ell}^i:=\lambda_i^{\prime\prime}(\xbf_{\ell})$, for $1\leq\ell\leq N$ and $1\leq i\leq \kappa$.
Then, the difference $\fos^{(\tsind)}-\ros^{(\tsind)}$ between the time-discrete solution of the full order model \eqref{eq:SV-full} and that of the reduced order model \eqref{eq:evoleq} at time $t^{\tsind}$ can be approximated by
\begin{equation}\label{eq:diff_approx}
    \errappr^{(\tsind)}:=-\left(\frac{\partial\resid_{\fom}^{(\tsind)}}{\partial\fos^{(\tsind)}}\bigg|_{(\ros^{(\tsind)},\ros^{(\tsind-1)})}\right)^{-1}\left(\resid_{\rom}^{(\tsind)}+\frac{\partial\resid_{\fom}^{(\tsind)}}{\partial\fos^{(\tsind-1)}}\bigg|_{(\ros^{(\tsind)},\ros^{(\tsind-1)})}
    \errappr^{(\tsind-1)}
    \right),
\end{equation}
where $\errappr^{(0)}=\fos^{(0)}-\ros^{(0)}$. 
Owing to \eqref{eq:inv},
the computation of $\errappr^{(\tsind)}$ does not require in practice the solution of a $2N\times2N$ linear system.
Moreover, thanks to the structure of $\jacE$, its product with a generic vector can be implemented at a linear arithmetic complexity in $N$. In particular, for generic $\xbf,\ybf\in\Rbb^N$,
since $\nabla\Lambda(\xbf)$ has at most $\pdeg+1$ non-zero elements in each row, the product $\nabla\Lambda(\xbf)^\top \ybf$ is computed in $O(N\pdeg)$ operations.

The procedure discussed above can be repeated for any parameter $\prm_s$, for $s=1,\dots,p$. Then the matrix $\merrappr^{(\tsind)}\in\Rbb^{2N\times p}$, whose $s$th column is given by the quantity $\errappr^{(\tsind)}$ associated with the $s$th parameter $\prm_s$,
is an approximation of the difference $\fos^{(\tsind)}-\ros^{(\tsind)}\in\Rbb^{2N\times p}$. Therefore, we propose to compute the error indicator as an approximation of the relative error, in the Frobenius norm, between the full order solution and the reduced order solution at time $t^{\tsind}$, that is
\begin{equation}\label{eq:Errind}
    \errind^{(\tsind)} := \dfrac{\normF{\merrappr^{(\tsind)}}}{\normF{\ros^{(\tsind)}+\merrappr^{(\tsind)}}}.
\end{equation}

Since the cost to compute the error indicator depends on the product of the number of particles $N$ and the number of test parameters $p$, we propose to only evaluate the approximated residual associated to $\spei$ sample parameters $\{\prm_{s_1},\dots,\prm_{s_{\spei}}\}$.
The error indicator $\errind^{(\tsind)}_{\star}$ is then obtained as in \eqref{eq:Errind} from the quantities $\rostar^{(\tsind)}\in\Rbb^{2N\times\spei}$ and
$\merrappr^{(\tsind)}_{\star}\in\Rbb^{2N\times \spei}$.
The set of $\spei$ sample parameters is constructed at the initial time and it is fixed throughout the simulation. In this work, we consider $\spei$ randomly selected sample parameters for simplicity. 

This procedure is summarized in \Cref{algo:RA-EI}.
\begin{algorithm}[H]
\caption{Computation of error indicator for rank adaptivity}\label{algo:RA-EI}
\normalsize
\begin{algorithmic}[1]
{\small
\Procedure{$(\errind_{\star}^{(\tsind)}, \merrappr_{\star}^{(\tsind)}, \mdiscref_{\star}^{(\tsind)}, \rostar^{(\tsind)})$=RA-EI}{$\merrappr_{\star}^{(\tsind-1)}$, $\mdiscref_{\star}^{(\tsind-1)}$, $\rostar^{(\tsind-1)}$, $\redb^{(\tsind)}$, $\cX^{(\tsind)}$, $\cV^{(\tsind)}$}
 \For{$j=1,\dots,\spei$}
 \State Reconstruct the approximate solution $\ros^{(\tsind)}(\prm_{s_j})$ associated to the $j$th sample parameter
 \State Compute the residual $\resid_{\rom}^{(\tsind)}(\prm_{s_j})=\resid(\ros^{(\tsind)}(\prm_{s_j}),\ros^{(\tsind-1)}(\prm_{s_j}))$ using \eqref{eq:full_residual}
 \State Compute the error approximation $\errappr^{(\tsind)}(\prm_{s_j})$ as in \eqref{eq:diff_approx}
 \State Store $\ros^{(\tsind)}(\prm_{s_j})$, $\mdiscref(\rX^{(\tsind)}(\prm_{s_j}))$ and $\errappr^{(\tsind)}(\prm_{s_j})$ as columns of the matrices $\rostar^{(\tsind)}$, $\mdiscref_{\star}^{(\tsind)}$ and $\merrappr_{\star}^{(\tsind)}$, respectively
 \EndFor
 \State Compute the error indicator $\errind^{(\tsind)}_{\star}=\displaystyle\frac{\normF{\merrappr_{\star}^{(\tsind)}}}{\normF{\rostar^{(\tsind)}+\merrappr_{\star}^{(\tsind)}}}$
 \EndProcedure}
\end{algorithmic}
\end{algorithm}
\noindent
The computational complexity of \Cref{algo:RA-EI} is $O(Nn\spei)+O(N\pdeg\spei)$.

For a given instance of the sample parameter, the approximated solution is first reconstructed at line 3 by multiplying the reduced basis matrix with the corresponding column of the coefficient matrix with complexity $O(Nn)$. 
At line 4, the computation of the residual requires the knowledge of the electric field $\mdiscref(\rX^{(\tsind)}(\prm_{s_j}))$: this has already been computed in the numerical solution of system \eqref{eq:evoleq}, so that the computation of $\resid_{\rom}^{(\tsind)}(\prm_{s_j})$ only involves sums of vectors, whose computational complexity is $O(N)$. 
Next, the error approximation $\errappr^{(\tsind)}(\prm_{s_j})$ is computed at line 5. The quantity $\dfrac{\partial\resid_{\fom}^{(\tsind)}}{\partial\fos^{(\tsind-1)}}\bigg|_{(\ros^{(\tsind)},\ros^{(\tsind-1)})}\errappr^{(\tsind-1)}$ appearing in \eqref{eq:diff_approx} is available from the previous time step owing to \eqref{eq:inv2}. Moreover, the matrices $\Lambda(\rX^{(\tsind)}(\prm_{s_j}))$ and $\nabla\Lambda(\rX^{(\tsind)}(\prm_{s_j}))$ required for the evaluation of the Jacobian $\jacE(\rX^{(\tsind)}(\prm_{s_j}))$, as prescribed by \eqref{eq:jacobian_E}, are already available from the computation of $\mdiscref$ in the solution of \eqref{eq:evoleq}. The only additional operations to be performed here are the construction of $\nabla^2\Lambda(\rX^{(\tsind)}(\prm_{s_j}))$, which has complexity $O(N\pdeg)$, and the multiplication of $\jacE(\rX^{(\tsind)}(\prm_{s_j}))$ with an $N$-dimensional vector, which requires $O(N\pdeg)$ operations as discussed above. 
The arithmetic complexity of the loop at lines 2--7 is therefore $O(Nn\spei)+O(N\pdeg\spei)$. Finally, the computation of the Frobenius norms at line 8 requires $O(N\spei)$ operations.

The error indicator in \eqref{eq:Errind} is used in the rank-adaptive approach to decide when the current approximation is no longer sufficiently accurate. If the error indicator becomes ``too large", then the reduced space is augmented by adding one column to the reduced basis matrix $\redb$. Following \cite{HPR23}, the rank update is performed if the error indicator satisfies
\begin{equation}\label{eq:update_criterion}
    \errind_\star^{(\tsind)} \geq \cone\ctwo^\mu\overline{\errind}_\star,
\end{equation}
where $\cone,\ctwo\in\Rbb$ are fixed parameters, $\overline{\errind}_\star$ is the value of the error indicator at the previous update and $\mu$ is the number of rank updates performed until time $t^\tsind$. The role of $\mu$ is to limit the frequency of the rank updates over time. As shown in \cite{HPR23}, this criterion is reliable and robust, with little sensitivity with respect to the values of $\cone$ and $\ctwo$. If the criterion is satisfied, then the rank update is performed and the new basis vector is defined as the singular vector of the matrix
$(I_{N}-\redb\redb^\top)\merrappr_{\star}^{(\tsind)}\in\Rbb^{N\times2\spei}$
corresponding to the largest singular value. In this way, the reduced space is augmented with the direction that is worst approximated at the time of the update. 

\subsection{Update of the coefficients of the reduced solution}\label{sec:coeff_upd}

Assuming that the reduced basis matrix $\redb^{\old}\in\Rbb^{N\times n}$ has been updated to $\redb^{\new}\in\Rbb^{N\times (n+1)}$ at a generic time $t^\tsind>0$ as described in the previous section, we now focus on the update of the coefficient matrices $\cX^{\old},\cV^{\old}\in\Rbb^{n\times p}$ to $\cX^{\new},\cV^{\new}\in\Rbb^{(n+1)\times p}$. We consider two different procedures in this work. In the first approach, the coefficient matrices are augmented with two rows of zeros, as proposed in \cite{HPR23}. 
This is equivalent to imposing that the adapted reduced solution $\ros^{\new}$ is the projection of $\ros^{\old}$, the reduced solution before rank adaptation, onto the updated reduced space, and it ensures that $\ros^{\old}=\ros^{\new}$. In order to motivate the second approach, we first observe that, ideally, one might want to set $\cX^{\new}=(\redb^{\new})^\top\fX^{(\tsind)}$ and $\cV^{\new}=(\redb^{\new})^\top\fV^{(\tsind)}$, where $\fos^{(\tsind)}=\begin{bmatrix}
    \fX^{(\tsind)} \\ \fV^{(\tsind)}
\end{bmatrix}$ is the time-discrete full order solution at time $t^\tsind$, as this would ensure that the updated reduced solution is the orthogonal projection of the full order solution onto the new reduced space. However, the full order solution $\fos^{(\tsind)}$ is not available. Since the matrix $\merrappr^{(\tsind)}\in\Rbb^{2N\times p}$ introduced in the previous section is an approximation of the difference $\fos^{(\tsind)}-\ros^{(\tsind)}$, a possible remedy would be to consider the quantity $\ros^{\old}+\merrappr^{(\tsind)}$ as a surrogate for the full order solution $\fos^{(\tsind)}$. On the other hand, computing $\merrappr^{(\tsind)}\in\Rbb^{2N\times p}$ would require to evaluate $\errappr^{(\tsind)}(\prm_s)$ for all test parameters $\prm_s$, $s=1,\dots,p$, and the computational complexity of this operation is $O(Np)$. 
Since $\errappr^{(\tsind)}$ has been computed for $\spei$ parameters in \Cref{algo:RA-EI}, we propose to approximate the values associated to the remaining parameters via interpolation. More precisely, for each test parameter $\prm_s$, we consider the approximation
\begin{equation}\label{eq:interp}
    \errappr^{(\tsind)}(\prm_s)\approx\errapprt^{(\tsind)}(\prm_s) := \sum_{i=1}^{\nti}\coeffint_i\sbfnt_i(\prm_s), \qquad \forall s=1,\dots,p,
\end{equation}
where $\sbfnt_i:\prms\to\Rbb$ are prescribed functions and $\coeffint_i\in\Rbb^{2N}$ for all $i=1,\dots,\nti$. This can equivalently be written as
\begin{equation*}
    \merrappr^{(\tsind)} \approx \merrapprt^{(\tsind)}=\mcoeffint\msbfnt,
\end{equation*}
where $\mcoeffint\in\Rbb^{2N\times\nti}$ is the matrix whose columns are the vectors $\coeffint_i$, for $i=1,\dots,\nti$, and $\msbfnt\in\Rbb^{\nti\times p}$ is the matrix whose $(i,s)$th entry is $\sbfnt_i(\prm_s)$. Next, we impose that the approximation is exact at $\spei$ sample parameters, that is, $\errappr^{(\tsind)}(\prm_{s_j})=\errapprt^{(\tsind)}(\prm_{s_j})$ for $j=1,\dots,\spei$. By introducing the matrix $\spmat\in\Rbb^{p\times\spei}$ whose $j$th column, for $j=1,\dots,\spei$, is the $s_j$th element of the canonical basis of $\Rbb^{p}$, gives $\merrappr^{(\tsind)}\spmat = \mcoeffint\msbfnt\spmat$.
The matrix $\mcoeffint$ can then be computed as
$\mcoeffint = \merrappr^{(\tsind)}\spmat(\msbfnt\spmat)^\dagger$, with pseudoinverse $(\msbfnt\spmat)^\dagger = (\msbfnt\spmat)^\top\big(\msbfnt\spmat(\msbfnt\spmat)^\top\big)^{-1}$.
Notice that a necessary condition for $\msbfnt\spmat\in\Rbb^{\nti\times\spei}$ to admit a pseudoinverse is that $\spei\geq\nti$. Finally we obtain 
\begin{equation*}
    \merrapprt^{(\tsind)} = \merrappr^{(\tsind)}\spmat(\msbfnt\spmat)^\dagger\msbfnt =: \merrappr_{\star}^{(\tsind)}\Psbfnt,
\end{equation*}
where $\Psbfnt=(\msbfnt\spmat)^\dagger\msbfnt\in\Rbb^{\spei\times p}$
and $\merrappr_{\star}^{(\tsind)}=\merrappr^{(\tsind)}\spmat\in\Rbb^{2N\times\spei}$.
In this work, we assume the basis functions $\{\sbfnt_i\}_{i=1}^{\nti}$ to be fixed so that the matrix $\Psbfnt$ can be precomputed, and the evaluation of $\merrapprt^{(\tsind)}$ only requires the computation of $\errappr^{(\tsind)}$ at $\spei<p$ sample parameters. With this, we perform the approximation $\fos^{(\tsind)}\approx\ros^{(\tsind)}+\merrappr_{\star}^{(\tsind)}\Psbfnt$ and we define
\begin{equation*}
    \cX^{\new}=(\redb^{\new})^\top(\redb^{\old}\cX^{\old}+\merrappr_{\star,X}^{(\tsind)}\Psbfnt) \qquad \text{ and } \qquad \cV^{\new}=(\redb^{\new})^\top(\redb^{\old}\cV^{\old}+\merrappr_{\star,V}^{(\tsind)}\Psbfnt).
\end{equation*}
If $A_{\redb}^{\new}\in\mathbb{R}^{2N\times 2n}$ is the block diagonal matrix having the two diagonal blocks equal to $\redb^{\new}\in\mathbb{R}^{N\times n}$, the rank update reads
\begin{equation}\label{eq:thetanew}
    \ros^{\new} = \ros^{\old} + \gamma A_{\redb}^{\new}(A_{\Psi}^{\new})^\top\merrappr_{\star}^{(\tsind)}\Psbfnt.
\end{equation}
Notice that, with the choice $\gamma=1$, $\ros^{\old}\neq\ros^{\new}$ in general, and the quality of the updated reduced solution improves, as long as the error indicator $\merrappr^{(\tsind)}$ is an accurate approximation of the difference between the full order solution and the reduced solution, and the interpolation error associated to \eqref{eq:interp} is sufficiently small.
\begin{proposition}
    Let 
    \begin{equation*}
        \merrappr_{\rm{ind}}^{(\tsind)}:=\fos^{(\tsind)}-\ros^{\old}-\merrappr^{(\tsind)} \quad \text{ and } \quad \merrappr_{\rm{interp}}^{(\tsind)}:=\merrappr^{(\tsind)}-\merrappr_\star^{(\tsind)}\Psbfnt.
    \end{equation*}
    If $\gamma=1$ in \eqref{eq:thetanew} and 
    \begin{equation}\label{eq:EindEinterp}        \normF{\merrappr_{\rm{ind}}^{(\tsind)}+\merrappr_{\rm{interp}}^{(\tsind)}} < \frac{1}{2}\normF{(A_{\redb}^{\new})^\top\merrappr_{\star}^{(\tsind)}\Psbfnt},
    \end{equation}
    then 
    \begin{equation*}
        \normF{\fos^{(\tsind)}-\ros^{\new}} < \normF{\fos^{(\tsind)}-\ros^{\old}}.
    \end{equation*}
\end{proposition}
\begin{proof}
    We first observe that
    \begin{equation*}
        \normF{\fos^{(\tsind)}-\ros^{\new}}^2 = \normF{\fos^{(\tsind)}-\ros^{\old}}^2 + \normF{\ros^{\old}-\ros^{\new}}^2 + 2\big\langle \fos^{(\tsind)}-\ros^{\old},\ros^{\old}-\ros^{\new}\big\rangle_F,
    \end{equation*}
    where $\langle\cdot,\cdot\rangle_F$ denotes the Frobenius inner product.
    For $\gamma=1$, using \eqref{eq:thetanew}
    and
    \begin{equation*}
        \fos^{(\tsind)}-\ros^{\old} = \merrappr_{\text{ind}}^{(\tsind)} + \merrappr_{\text{interp}}^{(\tsind)} + \merrappr_\star^{(\tsind)}\Psbfnt,
    \end{equation*}
    yields
    \begin{equation*}
        \normF{\fos^{(\tsind)}-\ros^{\new}}^2 = \normF{\fos^{(\tsind)}-\ros^{\old}}^2 -
        \norm{(A_{\Psi}^{\new})^\top\merrappr_{\star}^{(\tsind)}\Psbfnt}_F^2
        -2\langle \merrappr_{\text{ind}}^{(\tsind)}+\merrappr_{\text{interp}}^{(\tsind)},
           A_{\Psi}^{\new}(A_{\Psi}^{\new})^\top\merrappr_{\star}^{(\tsind)}\Psbfnt\rangle_F.
    \end{equation*}
    If the inner product appearing in the last term is non-negative, the result follows straightforwardly. If the last term is negative, then
    applying the Cauchy-Schwarz inequality together with \eqref{eq:EindEinterp} yields the result.
\end{proof}

The procedure for rank adaptivity is summarized in \Cref{algo:rank-update}. The parameter $\gamma\in\{0,1\}$ allows to unify the two strategies we presented in this section: we recover the method of \cite{HPR23} for $\gamma=0$, and our novel interpolation-based approach for $\gamma=1$. Assuming that $2\spei<N$, the arithmetic complexity is $O(N{\spei}^2)+O(Nn\spei)$ when $\gamma=0$, and $O(N{\spei}^2)+O(Nn\spei)+O(pn\spei)$ when $\gamma=1$. The only extra operations required in the second case are the matrix multiplications $(\redb^{\new})^\top\merrappr_{\star,X}^{(\tsind)}\Psbfnt$ and $(\redb^{\new})^\top\merrappr_{\star,V}^{(\tsind)}\Psbfnt$, since the quantity $\merrappr_{\star}^{(\tsind)}$ is already available from \Cref{algo:RA-EI}.
\begin{algorithm}
\caption{Rank update}\label{algo:rank-update}
\normalsize
\begin{algorithmic}[1]
{\small
\Procedure{$(\redb^{\new}, \cX^{\new}, \cV^{\new})=$rank\_update}{$\redb^{\old}, \cX^{\old}, \cV^{\old}, \merrappr_{\star}^{(\tsind)}$,$\gamma\in\{0,1\}$,$\Psbfnt$}
\State $\bm{\psi} \gets$ first left singular vector of
$(I_{N}-\redb^{\old}(\redb^{\old})^\top)\left[\merrappr_{\star,X}^{(\tsind)} \quad \merrappr_{\star,V}^{(\tsind)}\right]$
\State $\redb^{\new} \gets \begin{bmatrix} \redb^{\old} & \bm{\psi}\end{bmatrix}$
\State $\cX^{\new}\gets\begin{bmatrix} \cX^{\old} \\ \mathbf{0}_p^\top \end{bmatrix}+\gamma(\redb^{\new})^\top\merrappr^{(\tsind)}_{\star,X}\Psbfnt$\, and\, $\cV^{\new}\gets\begin{bmatrix} \cV^{\old} \\ \mathbf{0}_p^\top \end{bmatrix}+\gamma(\redb^{\new})^\top\merrappr^{(\tsind)}_{\star,V}\Psbfnt$
\EndProcedure}
\end{algorithmic}
\end{algorithm}

Here we focused on the case where the dimension of the reduced space is to be increased in order to accommodate for a growth of the numerical error, as this is most common in applications. If necessary, rank reduction can simply be achieved by removing columns from the reduced basis matrix associated with directions that have become redundant to describe the dynamics. Moreover, \Cref{algo:rank-update} can be easily generalized to the case where more than one basis function is added to the reduced space.

\section{Hyper-reduction of nonlinear terms}\label{sec:hyper-reduction}
As shown in \Cref{sec:ROM}, the solution of the reduced dynamics \eqref{eq:rom} has a computational cost that still depends on the product of the number of particles $N$ and the number of test parameters $p$ despite dimensionality reduction. For this reason, solving the ROM can in principle be more computationally demanding than solving the FOM. In this section, we propose a hyper-reduction strategy to decouple the operations that depend on $N$ from those that depend on $p$.

The main contribution to the total computational cost can be traced back to two operations. First, the nonlinear, reduced electric field $\redb^\top\mdiscref(\redb\cX)$ has to be computed in \eqref{eq:ZVdot} and \eqref{eq:Phidot}. Although this quantity is a $n\times p$ matrix, this operation requires to evaluate the particle-to-grid map for all $N$ macro-particles. Second, the full order electric field $\mdiscref(\redb\cX)\in\Rbb^{N\times p}$ is still involved in \eqref{eq:Phidot}. We address the second issue first.

\subsection{Parameter subsampling for the electric field}
Observe first that the reason why the electric field associated to all test parameters is needed is that the matrix multiplication with $\cV^\top\in\Rbb^{p\times n}$ is required to evolve the reduced basis in \eqref{eq:Phidot}. As in \cite{PV25}, the idea to alleviate the computational burden associated to this operation is to carry out the matrix multiplication by selecting a subsample of parameters. In particular, note that
\begin{equation}\label{eq:ef_times_Wt}
    \mdiscref\big(\redb(t)\cX(t)\big)\cV^\top(t)
    = \sum_{s=1}^p\nabla\nlham\big(\redb(t)\rposbf(t,\prm_s)\big)\rvelbf(t,\prm_s)^\top.
\end{equation}
We propose to approximate the average over all test parameters using a subset of $\spav<p$ sample parameters $\{\prm_{s_1},\dots,\prm_{s_{\spav}}\}$. Then, we approximate the product in \eqref{eq:ef_times_Wt} as
\begin{equation}\label{eq:psample}
    \mdiscref\big(\redb(t)\cX(t)\big)\cV(t)^\top \approx
    \dfrac{p}{\spav}\mdiscref\big(\redb(t)\cXs(t)\big)\cVs^\top(t)
    := \dfrac{p}{\spav}\sum_{r=1}^{\spav}\nabla\nlham\big(\redb(t)\rposbf(t,\prm_{s_r})\big)\rvelbf(t,\prm_{s_r})^\top
\end{equation}
where $\mdiscref(\redb(t)\cXs(t)),\cVs(t)\in\Rbb^{N\times \spav}$. In this way, it is only required to compute the electric field associated to $\spav$ sample parameters, and the total complexity to evolve the basis in \eqref{eq:Phidot} is reduced from $O(Npn)$ to $O(N\spav n)+O(Nn^2)$. As proposed in \cite{PV25}, the sample parameters are determined based on a QR factorization of the coefficient matrix $\coef(t)\in\Rbb^{2n\times p}$ with column pivoting, whose computational complexity is $O(pn\spav)$. The rationale behind this choice is to select the parameters that are ``most relevant'' for the dynamics.

\subsection{Hyper-reduction via approximation of the particle-to-grid map}
Next, we discuss how to address the evaluation of the reduced electric field $\redb^\top \mdiscref(\redb \cX)$ in the evolution of the particles' positions coefficients \eqref{eq:ZVdot}, which still requires the computation of the particle-to-grid map for all computational particles. The arithmetic complexity of this operation might become prohibitively high, especially when it is performed at each time step for a large number of test parameters.
In this section we present one of the main contributions of this work, namely an extension of the framework of \cite{PV23,PV25} to non-sparse Hamiltonians. In the following we assume that time is fixed, and we omit time dependency for simplicity of notation.

We first remark that, for any $\rposbf\in\Rbb^n$ and $\redb\in\Rbb^{N\times n}$, the nonlinear part \eqref{eq:nonlinear_Ham} of the reduced Hamiltonian can be written as
the sum of $O(\fedim^2)$ terms, where each function $\decdi_j$ in \eqref{eq:gj}, for $1\leq j\leq \kappa$, only depends on the particles contained in the support of the $j$th spatial basis function $\lambda_j$. Then, the idea is to diagonalize the inverse of the stiffness matrix $\stmat^{-1}\in\Rbb^{\fedim\times\fedim}$ to decouple the contributions of the $\fedim$ basis functions and reduce the number of terms in the decomposition of \eqref{eq:nonlinear_Ham} to $O(\fedim)$.
Since $\stmat\in\Rbb^{\kappa\times\kappa}$ is symmetric and positive definite, it can be diagonalized by an orthogonal matrix $V$ and its eigenvalues $\{\delta_\idx\}_{\idx=1}^{\kappa}$ are strictly positive.
We can then write $\nlrham$ in \eqref{eq:nonlinear_Ham} as
\begin{equation*}
    \nlrham(\rposbf;\redb) = \dfrac{N}{2\ell_x}\sum_{\idx=1}^{\fedim}\dfrac{1}{\delta_\idx}(V^{\top}\decd(\redb \rposbf))_\idx^2
    = \sum_{\idx=1}^{\fedim}\dfrac{N}{2\ell_x\delta_\idx}
    \bigg((V^{\top}\sbf)_\idx-\dfrac{\ell_x}{N}\big(V^{\top}\Lambda(\redb \rposbf)^\top\mathbf{1}_N\big)_\idx\bigg)^2.
\end{equation*}
Let us introduce the matrix-valued function
\begin{equation*}
\begin{aligned}
\mGfun: \Rbb^{N} &\quad \longrightarrow\quad \Rbb^{N\times\kappa}\\
\posbf &\quad  \longmapsto\quad \mGfun(\posbf)\quad\mbox{such that}\quad\mGfun(\xbf)^\idx_{\ell}=\sum_{j=1}^{\kappa}V_j^\idx\lambda_j(\pos_{\ell}).
\end{aligned}
\end{equation*}
Then, the nonlinear part of the reduced Hamiltonian can be written as
\begin{equation}\label{eq:Ham_dec}
    \nlrham(\rposbf;\redb) = 
    \sum_{\idx=1}^{\fedim}\left(\sqrt{\dfrac{N}{2\ell_x\delta_\idx}}
    \sbf^{\top} \Vbf^\idx-
    \sqrt{\dfrac{\ell_x}{2N\delta_\idx}}\sum_{\ell=1}^NG(\redb \rposbf)_{\ell}^\idx\right)^2
    =\sum_{\idx=1}^{\fedim}\left(\alpha_\idx + (\boldsymbol{\beta}^\idx)^\top \Gfun^\idx(\redb \rposbf)\right)^2.
\end{equation}
where $\Gfun^\idx:\xbf\in\Rbb^{N}\mapsto\Gfun^\idx(\xbf)\in\Rbb^N$ is such that $\Gfun^\idx(\xbf)_{\ell}=\mGfun(\xbf)^\idx_{\ell}$ according to the definition above, and
\begin{equation*}
    \alpha_\idx:=\sqrt{\frac{N}{2\ell_x\delta_\idx}}\sbf^\top \Vbf^\idx\in\Rbb,
    \quad \quad \boldsymbol{\beta}^\idx:=-\sqrt{\dfrac{\ell_x}{2N\delta_\idx}}\mathbf{1}_N\in\Rbb^N. 
\end{equation*}
For each $\idx=1,\dots,\fedim$, the product $(\boldsymbol{\beta}^\idx)^{\top}\Gfun^\idx(\redb \rposbf)$ is a linear combination of $N$ functions, each depending on the position of one macro-particle.
Using the chain rule, the gradient of $\nlrham$ with respect to $\rposbf$ can be written in terms of $\Gfun^\idx$ and its Jacobian matrix
$\jacG{\idx}\in\Rbb^{N\times N}$ as
\begin{equation}\label{eq:grad_red}
    \nabla_{\rposbf}\nlrham(\rposbf;\redb) = 2\sum_{\idx=1}^{\fedim}\left(\alpha_\idx+(\boldsymbol{\beta}^\idx)^{\top} \Gfun^\idx(\redb \rposbf)\right)\redb^\top \jacG{\idx}(\redb \rposbf)\boldsymbol{\beta}^\idx.
\end{equation}
Since, for any $\xbf\in\Rbb^{N}$, each entry of $\Gfun^\idx(\xbf)$ only depends on one component of the input vector $\xbf$, the Jacobian $\jacG{\idx}(\xbf)\in\Rbb^{N\times N}$ is diagonal
and $\jacG{\idx}(\posbf)_\ell^\ell=\sum_{j=1}^{\kappa}V_j^\idx\nabla \Lambda(\posbf)^j_{\ell}=\sum_{j=1}^{\kappa}V_j^\idx\lambda'_j(\pos_{\ell})$.

Following the approach outlined in \Cref{sec:non-sparse}, we now propose to approximate the vector-valued maps $\Gfun^\idx$ using EIM \cite{BMNP04}. For all $\idx=1,\dots,\fedim$, let $\deimb{\idx}\in\Rbb^{N\times m_\idx}$ be a matrix whose columns span a $m_\idx$-dimensional subspace of $\Rbb^N$, with $m_\idx\geq 1$. Given a set of interpolation indices $\{s^\idx_1,\dots,s^\idx_{m_\idx}\}\subset\{1,\dots,N\}$, we define the matrix $\deimi{\idx}\in\Rbb^{N\times m_\idx}$ whose $j$th column, for $j=1,\dots,m_\idx$, is the $s_j^\idx$ element of the canonical basis of $\Rbb^N$.
Then, the EIM approximation of $\Gfun^\idx$ reads
\begin{equation*}
    \Gfun^\idx(\posbf) \approx \Pbb^\idx \Gfun^\idx(\posbf),\qquad\forall\, \posbf\in\Rbb^{N},
\end{equation*}
%
where $\Pbb^\idx=\deimb{\idx}((\deimi{\idx})^\top \deimb{\idx})^{-1}(\deimi{\idx})^\top\in\Rbb^{N\times N}$ is the EIM projection matrix. 
Replacing this approximation into \eqref{eq:Ham_dec} yields the hyper-reduced nonlinear function
\begin{equation}\label{eq:Ham_hred}
    \nlhrham(\rposbf;\redb) := \sum_{\idx=1}^\fedim\left(\alpha_\idx + (\boldsymbol{\beta}^\idx)^{\top}\Pbb^\idx \Gfun^\idx(\redb \rposbf)\right)^2 = \sum_{\idx=1}^\fedim\left(\alpha_\idx + (\widehat{\boldsymbol{\beta}}^\idx)^\top\widehat{\Gfun}^\idx(\redb \rposbf)\right)^2
\end{equation}
where $\widehat{\boldsymbol{\beta}}^\idx := ((\deimb{\idx})^\top \deimi{\idx})^{-1}(\deimb{\idx})^\top \boldsymbol{\beta}^\idx\in\Rbb^{m_\idx}$ and
$\widehat{\Gfun}^\idx:\xbf\in\Rbb^{N}\mapsto\widehat{\Gfun}^\idx(\xbf)\in\Rbb^{m_\idx}$ is defined as
$\widehat{\Gfun}^\idx(\xbf)_\ell=\Gfun^\idx(\xbf)_{s^\idx_{\ell}}$,
with $1\leq\ell\leq m_{\idx}$.
%
Analogously, we define the hyper-reduced Hamiltonian
\begin{equation}\label{eq:hred_Ham}
    \hrham(\rposbf,\rvelbf;\redb) = \dfrac12 \rvelbf^\top \rvelbf + \nlhrham(\rposbf;\redb) \qquad \forall\rposbf,\rvelbf\in\Rbb^n,\redb\in\Rbb^{N\times n}.
\end{equation}
Notice that the quadratic part of the reduced Hamiltonian corresponds to a linear term in the evolution equation, and, therefore, does not require any hyper-reduction. 

A simple computation shows that the gradient of the hyper-reduced Hamiltonian with respect to the coefficients $\rposbf$ coincides with the hyper-reduced discrete electric field, namely
\begin{equation}\label{eq:grad_hred}
    \nabla_{\rposbf}\nlhrham(\rposbf;\redb) = 2\sum_{\idx=1}^{\fedim}\left(\alpha_\idx+(\widehat{\boldsymbol{\beta}}^\idx)^\top \widehat{\Gfun}^\idx(\redb \rposbf)\right)\redb^\top \left(\jacGhat{\idx}(\redb\rposbf)\right)^\top\widehat{\boldsymbol{\beta}}^\idx\in\Rbb^n,
\end{equation}
where, for any $\xbf\in\Rbb^N$, the entries of $\jacGhat{\idx}(\xbf)\in\Rbb^{m_\idx\times N}$ are
\begin{equation*}
    \jacGhat{\idx}(\xbf)_i^j = \begin{cases}
        \jacG{\idx}(\xbf)_{s_i^\idx}^j & \text{if } j\in\{s_1^\idx,\dots,s_{m_\idx}^\idx\} \\
        0 & \text{otherwise}
    \end{cases}.
\end{equation*}
Therefore, in order to compute the product $\redb^\top\left(\jacGhat{\idx}(\redb \rposbf)\right)^\top\in\Rbb^{n\times m_\idx}$ for a fixed $\idx$, it is sufficient to evaluate the Jacobian $\jacG{\idx}$ only at the $m_\idx$ entries associated to the interpolation indices. 

\subsection{The hyper-reduced model}
Replacing the hyper-reduced Hamiltonian \eqref{eq:hred_Ham} into the evolution equation \eqref{eq:ZVdot} for the reduced particles' velocity and the particle sampling \eqref{eq:psample} into the evolution \eqref{eq:Phidot} of the reduced basis, we end up with the system
\begin{subequations}\label{eq:hrom}
\begin{empheq}[left=\empheqlbrace]{align}
   & \dcX(t) = \cV(t) \label{eq:ZXhrdot} \\
   & \dcV(t) = -\mhrdiscref\big(\redb(t)\cX(t)\big) \label{eq:ZVhrdot}
   \\ 
    & \Dot{\redb}(t)=\big(\redb(t)\redb(t)^\top-I_N\big)\mdiscref\big(\redb(t) \cXs(t)\big)\cVs(t)^\top M^{-1}\big(\cX(t),\cV(t)\big),
   \label{eq:Phihrdot}
\end{empheq}
\end{subequations}
where $\mhrdiscref(\redb\cX)\in\Rbb^{n\times p}$ is the matrix with entries
$\mhrdiscref(\redb\cX)_{\ell}^s=\partial_{y_{\ell}}\nlhrham(\cXbf^s;\redb)$.
We refer to this formulation as the hyper-reduced model (hROM).
Its numerical solution using the time integration scheme proposed in \Cref{sec:pRK}
is summarized in \Cref{algo:PRK2-hr}, where one time step is considered.

In order to assemble the gradient \eqref{eq:grad_hred} of the hyper-reduced Hamiltonian, the function $\Gfun^\idx$ and its Jacobian $\jacG{\idx}$, for any $1\leq\idx\leq\kappa$, need to be evaluated only at $m_{\idx}$ entries of the approximate state $\redb\rposbf$.
Notice that the EIM approximations $\Pbb^{\idx}\Gfun^{\idx}$ of $\Gfun^\idx$, for different values of $\idx$, with $1\leq\idx\leq \kappa$, might share common interpolation indices. In this case, the particle-to-grid map has to be evaluated only once for each particle associated to a repeated index. Let $m\leq\sum_{\idx=1}^{\fedim}m_\idx$ be the number of unique interpolation indices. Then, computing the particle-to-grid map for all $\idx=1,\dots,\fedim$ has complexity $O(m\pdeg)$, and the remaining operations for the computation of the gradient of the nonlinear part of the Hamiltonian \eqref{eq:grad_hred} have complexity $O(nm)$. Since this procedure is carried out for all test parameters, the total complexity is $O(pmn)+O(pm\pdeg)$.

\begin{algorithm}
\caption{One time step of the partitioned Runge-Kutta scheme applied to the hyper-reduced model}\label{algo:PRK2-hr}
\normalsize
\begin{algorithmic}[1]
{\small
\Procedure{$(\redb^{(1)}, \cX^{(1)}, \cV^{(1)})=$PRK-HR}{$\redb^{(0)}, \cX^{(0)}, \cV^{(0)}, \{\deimb{\idx}\}_{\idx=1}^\fedim, \{\deimi{\idx}\}_{\idx=1}^\fedim$}
 \State $\cV^{(1/2)}=\cV^{(0)}-\dfrac{\Delta t}{2}\mhrdiscref(\redb^{(0)}\cX^{(0)})$
 \State Select $\spav$ sample parameters via QR factorization of $\begin{bmatrix}
     \cX^{(0)} \\ \cV^{(1/2)}
 \end{bmatrix}$ with column pivoting
 \State Construct $\cXs^{(0)}$ and $\cVs^{(1/2)}$ by selecting the columns of $\cX^{(0)}$ and $\cV^{(1/2)}$ associated to the $\spav$ sample parameters
 \State $\prpsi^{(1/2)}=-\Delta t\displaystyle\frac{p}{\spav}\mathcal{L}(\cXs^{(0)},\cVs^{(1/2)},\redb^{(0)})M^{-1}(\cX^{(0)},\cV^{(1/2)})$
 \State $\cX^{(1)}=\cX^{(0)}+\Delta t\,\cV^{(1/2)}$
 \State $\redb^{(1/2)}=\retr_{\redb^{(0)}}(\prpsi^{(1/2)})$
 \State Select $\spav$ sample parameters via QR factorization of $\begin{bmatrix}
     \cX^{(1)} \\ \cV^{(1/2)}
 \end{bmatrix}$ with column pivoting
 \State Construct $\cXs^{(1)}$ and $\cVs^{(1/2)}$ by selecting the columns of $\cX^{(1)}$ and $\cV^{(1/2)}$ associated to the $\spav$ sample parameters
\State $\prpsi^{(1)}=\displaystyle\frac{\prpsi^{(1/2)}}{2}-\displaystyle\frac{\Delta t}{2}\displaystyle\frac{p}{\spav}({d\retr_{\redb^{(0)}}}[\prpsi^{(1/2)}])^{-1}
\mathcal{L}(\cXs^{(1)},\cVs^{(1/2)},\redb^{(1/2)})M^{-1}(\cX^{(1)},\cV^{(1/2)})$
 \State $\cV^{(1)}=\cV^{(1/2)}-\displaystyle\frac{\Delta t}{2}\mhrdiscref(\redb^{(1/2)}\cX^{(1)})$
 \State $\redb^{(1)}=\retr_{\redb^{(0)}}(\prpsi^{(1)})$
 \EndProcedure}
\end{algorithmic}
\end{algorithm}

The arithmetic complexity of \Cref{algo:PRK2-hr} is $O(Nn^2)+O(N\spav n)+O(p\spav n)+O(pmn)+O(pm\pdeg)$ and is therefore independent of the product of the number $N$ of macro-particles and the number $p$ of test parameters.

\subsection{Error bound on the EIM approximation}
The quality of the EIM approximation of the reduced gradient described in the previous section is quantified by the following result.
\begin{proposition}\label{prop:error_bound_EIM}
Let $\rposbf\in\Rbb^n$ and $\redb\in\Rbb^{N\times n}$.
Let $\nlrham$ denote the reduced Hamiltonian defined in \eqref{eq:nonlinear_Ham} and let $\nlhrham$ be the hyper-reduced Hamiltonian defined in \eqref{eq:Ham_hred}. Then,
    \begin{equation*}
        \norm{\nabla_{\rposbf}\nlrham(\rposbf;\redb)-\nabla_{\rposbf}\nlhrham(\rposbf;\redb)}\leq K_1\sum_{\idx=1}^{\fedim}\norm{(I-\Pbb^\idx)F^\idx(\redb \rposbf;\redb)}_2+K_2\sum_{\idx=1}^{\fedim}\norm{(I-\Pbb^\idx)\Gfun^\idx(\redb \rposbf)}
    \end{equation*}
    where, for any $\posbf\in\Rbb^{N}$,
        $F^\idx(\posbf;\redb) := \left(\alpha_\idx+(\boldsymbol{\beta}^\idx)^{\top} \Gfun^\idx(\xbf)\right)\jacG{\idx}(\posbf)\redb\in\Rbb^{N\times n}$.
    The constants $K_1$ and $K_2$ are given by
    \begin{equation*}
        K_1=\sqrt{\frac{2\ell_x}{\displaystyle\min_{1\leq \idx\leq\kappa} \delta_\idx}}, \qquad K_2 = \frac{\ell_x}{\sqrt{N}}\max_{1\leq \idx\leq \kappa}\left(\delta_\idx^{-1}\norm{((\deimb{\idx})^\top \deimi{\idx})^{-1}(\deimb{\idx})^\top\mathbf{1}_N}\max_{x\in\Omega_x}\bigg\lvert\sum_{j=1}^{\kappa}V_j^\idx\lambda_j'(x)\bigg\rvert\right).
    \end{equation*} 
\end{proposition}
\begin{proof}
    Starting from the definition of the reduced gradient \eqref{eq:grad_red} and of the hyper-reduced gradient \eqref{eq:grad_hred}, subtracting and adding the same quantity
    \begin{equation*}
        2\sum_{\idx=1}^{\fedim}\left(\alpha_\idx +(\boldsymbol{\beta}^\idx)^{\top} \Gfun^\idx(\redb \rposbf)\right)\redb^\top \jacG{\idx}(\redb \rposbf)(\Pbb^\idx)^\top \boldsymbol{\beta}^\idx
    \end{equation*}
    and using the triangle inequality, we get
    \begin{align*}
        \norm{\nabla_{\rposbf}\nlrham(\rposbf;\redb)-\nabla_{\rposbf}\nlhrham(\rposbf;\redb)}
        &\leq2\norm{\sum_{\idx=1}^{\fedim}\left(\alpha_\idx + (\boldsymbol{\beta}^\idx)^{\top} \Gfun^\idx(\redb \rposbf)\right)\redb^\top \jacG{\idx}(\redb \rposbf)\big(I-\Pbb^\idx\big)^\top \boldsymbol{\beta}^\idx}\\
        &+2\norm{\sum_{\idx=1}^{\fedim}\left((\boldsymbol{\beta}^\idx)^{\top}(I-\Pbb^\idx)\Gfun^\idx(\redb \rposbf)\right)\redb^\top \jacG{\idx}(\redb \rposbf)(\Pbb^\idx)^\top \boldsymbol{\beta}^\idx}\\&\leq2\sum_{\idx=1}^{\fedim}\norm{\boldsymbol{\beta}^\idx}\norm{(I-\Pbb^\idx)F^\idx(\redb \rposbf;\redb)}_2\\&+2\sum_{\idx=1}^{\fedim}\norm{\boldsymbol{\beta}^\idx}\norm{\redb^\top \jacG{\idx}(\redb \rposbf)(\Pbb^\idx)^\top \boldsymbol{\beta}^\idx}\norm{(I-\Pbb^\idx)\Gfun^\idx(\redb \rposbf)}.
    \end{align*}
    We next observe that $\norm{\boldsymbol{\beta}^\idx}=\sqrt{\ell_x(2\delta_\idx)^{-1}}$ and that
    \begin{align*}
        \norm{\redb^\top \jacG{\idx}(\redb \rposbf)(\Pbb^\idx)^\top \boldsymbol{\beta}^\idx}
        &\leq\norm{\redb^\top}_2\norm{\jacG{\idx}(\redb \rposbf)}_2\norm{\deimi{\idx}}_2\norm{((\deimb{\idx})^\top \deimi{\idx})^{-1}(\deimb{\idx})^\top \boldsymbol{\beta}^\idx}\\&=\sqrt{\frac{\ell_x}{2N\delta_\idx}}\norm{\jacG{\idx}(\redb \rposbf)}_2\norm{((\deimb{\idx})^\top \deimi{\idx})^{-1}(\deimb{\idx})^\top\mathbf{1}_N}\\
         &\leq\sqrt{\dfrac{\ell_x}{2N\delta_\idx}}\norm{((\deimb{\idx})^\top \deimi{\idx})^{-1}(\deimb{\idx})^\top\mathbf{1}_N}\max_{1\leq\ell\leq N}\left\lvert\jacG{\idx}(\redb \rposbf)_\ell^\ell\right\rvert.
    \end{align*}
    Finally, we have
    \begin{align*}
        \max_{1\leq\ell\leq N}\left\lvert\jacG{\idx}(\redb \rposbf)_\ell^\ell\right\rvert= \max_{1\leq\ell\leq N}\bigg\lvert\sum_{j=1}^{\kappa}V_j^\idx\lambda'_j(\pos_{\ell})\bigg\rvert \leq \max_{x\in\Omega_x}\bigg\lvert\sum_{j=1}^{\kappa}V_j^\idx\lambda'_j(x)\bigg\rvert
    \end{align*} 
    which gives $K_2$.
\end{proof}
An explicit expression of the constants $K_1$ and $K_2$ appearing in the error bound is derived in \ref{sec:k1} in the case of a one-dimensional domain discretized using a uniform grid and continuous, piecewise linear basis functions.

\begin{remark}
    The result of \Cref{prop:error_bound_EIM} can be extended to decompositions of the form
    \begin{equation*}
        \nlrham(\rposbf;\redb) = \sum_{\idx=1}^D\mathscr{F}(\alpha_\idx+(\boldsymbol{\beta}^\idx)^{\top} \Gfun^\idx(\redb\rposbf))
    \end{equation*}
    with general $\mathscr{F}:\Rbb\to\Rbb$ only satisfying mild assumptions. The decomposition of the Hamiltonian in the Vlasov-Poisson problem \eqref{eq:Ham_dec} corresponds to $D=\fedim$ and $\mathscr{F}(x)=x^2$. In this sense, this result is an extension of the work in \cite{PV23,PV25}, which was limited to the case $D=1$ and $\mathscr{F}(x)=x$.
\end{remark}

\subsection{Construction of the EIM space}
\Cref{prop:error_bound_EIM} suggests that an accurate approximation of the reduced gradient is achieved by minimizing the projection errors of $F^\idx$ and $\Gfun^\idx$ associated to the EIM projection matrix $\Pbb^\idx$. In light of this result, in order to construct the EIM basis $\deimb{\idx}$ and the EIM interpolation indices $\deimi{\idx}$ for all $\idx\in\{1,\dots,\fedim\}$, we assemble the snapshot matrix at time $t$
\begin{equation}\label{eq:EIMsm}
    \Scal^\idx(t) = \begin{bmatrix}
        \Scal^\idx_\Gfun(t) & \Scal^\idx_{F}(t)
    \end{bmatrix}\in\Rbb^{N\times(n+1)\spdb}
\end{equation}
associated to $\spdb$ sample parameters $\{\prm_{s_1},\dots,\prm_{s_{\spdb}}\}\subset\{\prm_1,\dots,\prm_p\}$, where
\begin{align*}
    \Scal^\idx_\Gfun(t) &= \begin{bmatrix}
        \Gfun^\idx(\redb(t) \rposbf(t,\prm_{s_1})) & \dots & \Gfun^\idx(\redb(t) \rposbf(t,\prm_{s_{\spdb}}))
    \end{bmatrix} \in \Rbb^{N\times\spdb} \\
    \Scal^\idx_{F}(t) &= \begin{bmatrix}
        F^\idx(\redb(t)\rposbf(t, \prm_{s_1});\redb(t)) & \dots & F^\idx(\redb(t) \rposbf(t,\prm_{s_{\spdb}});\redb(t))
    \end{bmatrix} \in \Rbb^{N\times n\spdb}.
\end{align*}
The EIM basis matrix and interpolation indices at the initial time can be constructed from snapshots of the initial condition. Then, the EIM approximation is updated over time during the simulation. In \cite{PV25}, this operation is performed via low-rank updates of a small subsample of rows of the EIM basis matrix, following the approach of \cite{Peh20}. In this work, we employ a greedy algorithm to reconstruct the EIM basis at every adaptation step. The main reason behind this choice is that the solutions of the Vlasov-Poisson problem lack coherent structures that are local in space. In this case, the adaptive algorithm of \cite{PV25} might require to update all rows of the EIM basis to keep the approximation error sufficiently low, leading to a large computational cost (see \cite[Section 8.2]{PV25}). Moreover, reconstructing the EIM basis allows to automatically update the dimension of the EIM space, which is useful in situations where the numerical rank of the nonlinear term evolves over time.

More in detail, we rely on the greedy technique outlined in \Cref{algo:gEIM}.

\begin{algorithm}
\caption{Greedy construction of the EIM basis and interpolation indices}\label{algo:gEIM}
\normalsize
\begin{algorithmic}[1]
{\small
\Procedure{$(\{\deimb{\idx}\}_{\idx=1}^\fedim,\{\deimi{\idx}\}_{\idx=1}^\fedim)$=greedyEIM}{$\redb$, $\cX$, tol}
\For{$\idx=1,\dots,\fedim$}
\State Construct the snapshot matrix $\Scal^\idx$ as in \eqref{eq:EIMsm}
\State $R=[\Rbf^1 \dots \Rbf^{(n+1)\spdb}]\gets\Scal^\idx$
\State Select $j_1$ such that $\norm{\Rbf^{j_1}}=\max_{1\leq l\leq(n+1)\spdb}\norm{\Rbf^l}$
\State $m\gets1$, $\deimb{\idx}\gets[\quad]$, $\deimi{\idx}=\emptyset$
\While{$\norm{\Rbf^{j_m}}>$ tol}
\State $\deimb{\idx}\gets [\deimb{\idx} \quad \Rbf^{j_m}]$
\State $\deimi{\idx}\gets \deimi{\idx}\cup\{\text{argmax}_l\lvert \Rbf^{j_m}_l\rvert\}$
\State $R\gets \Scal^\idx - \Pbb^{\idx}\Scal^\idx$
\State Select $j_{m+1}$ such that $\norm{\Rbf^{j_{m+1}}}=\max_{1\leq l\leq(n+1)\spdb}\norm{\Rbf^l}$
\State $m\gets m+1$
\EndWhile
\EndFor
\EndProcedure}
\end{algorithmic}
\end{algorithm}
For a given $\idx\in\{1,\dots,\fedim\}$, this algorithm computes the EIM basis and interpolation simultaneously with computational complexity $O(Nn^2\spdb)$. In particular, we require that $\spdb\ll p$, so that the cost of assembling the snapshot matrix $\Scal^\idx$ and constructing the EIM space does not spoil the overall complexity of the method. Observe that the snapshots can be assembled from quantities that are already available from the time evolution of the hyper-reduced model \eqref{eq:hrom}. In particular, we do not require any knowledge of the full order solution. We run \Cref{algo:gEIM} every $\eimuf\geq1$ time steps during the simulation to recompute the EIM approximation. The set of sample parameters is adapted based on a QR factorization of the coefficient matrix with column pivoting.
We remark that, although the process of reconstructing the EIM basis and interpolation points for all $\idx\in\{1,\dots,\fedim\}$ might become expensive for large values of $\fedim$, it is typically not the dominant operation in the solution of the hyper-reduced system, for several reasons. First, it is typically not necessary to run \Cref{algo:gEIM} at every time step, that is, setting $\eimuf>1$ is sufficient to keep the EIM approximation error under control. We refer to \Cref{sec:num_exp} for more details on this aspect. Second, every EIM projection can be computed independently, so that \Cref{algo:gEIM} can be easily parallelized. Finally, the number of computational particles and test parameters is always assumed to be larger than $\fedim$, which is in turn related to the number of spatial intervals and the polynomial degree of the finite element space.

\subsection{Summary of the algorithm}
The entire procedure for the numerical solution of the hROM is summarized in \Cref{algo:VP-hROM}. The input is the initial condition of the problem, a tolerance for the construction of the EIM basis, the frequency of the EIM updates $\eimuf\geq 1$ and the control parameters $C_1$, $C_2$ for rank adaptivity. The reduced basis at the initial time is constructed via cotangent lift of the initial condition \cite{PM16}, which ensures that the initial reduced basis matrix is block-diagonal, as discussed in \Cref{sec:ROM}. The coefficient matrices $\cX^{(0)}$ and $\cV^{(0)}$ are computed as the projection of $\fX^{(0)}$ and $\fV^{(0)}$ onto the reduced space.
Then, the initial EIM basis and interpolation indices are computed at line 5 using \Cref{algo:gEIM}. At lines 6 to 8, the error indicator is initialized by computing the projection error at the initial time associated to $\spei$ sample parameters. Next, at each time step, the reduced basis and coefficient matrices are first advanced according to \Cref{algo:PRK2-hr} using the current EIM approximation, and the error indicator is computed at line 12 following \Cref{algo:RA-EI}. If the value of the latter is large enough according to the control parameters $C_1$ and $C_2$, the rank of the reduced solution is updated as described in \Cref{algo:rank-update}. Finally, \Cref{algo:gEIM} is executed every $\eimuf$ time steps at line 19 to update the EIM approximation.

\begin{algorithm}[H]
\caption{VP-hROM}\label{algo:VP-hROM}
\normalsize
\begin{algorithmic}[1]
{\small
\Procedure{VP-hROM}{$\fX^{(0)}$, $\fV^{(0)}$, tol, $\eimuf$, $\cone$, $\ctwo$, $\gamma$}
\State Build $T\in\Rbb^{\fedim\times\fedim}$
\State Construct $\redb^{(0)}\in\Rbb^{N\times n}$ from $\begin{bmatrix}
    \fX^{(0)} & \fV^{(0)}
\end{bmatrix}$ via cotangent lift
\State $\cX^{(0)}\gets(\redb^{(0)})^\top\fX^{(0)}$, $\cV^{(0)}\gets(\redb^{(0)})^\top\fV^{(0)}$
\State $(\{\deimb{\idx}\}_{\idx=1}^\fedim,\{\deimi{\idx}\}_{\idx=1}^\fedim)=\textsc{greedyEIM}(\redb^{(0)},\cX^{(0)},\text{tol})$ as in \Cref{algo:gEIM}
\State Select $\spei$ parameters $\{\prm_{s_1},\dots,\prm_{s_{\spei}}\}$ for the computation of the error indicator
\State Compute the electric field $\mdiscref_\star^{(0)}$ associated to the sample parameters at the initial time
\State Compute the matrix $\Psbfnt$ as in \Cref{sec:coeff_upd}
\State $\merrappr^{(0)}_{\star}\gets \fostar^{(0)}-\rostar^{(0)}$, \quad $\overline{\errind}_\star\gets \normF{\merrappr^{(0)}_{\star}}\normF{\fostar^{(0)}}^{-1}$, \quad $\mu \gets 0$
\For{$\tsind=1,\dots,\nt$}
    \State $(\redb^{(\tsind)},\cX^{(\tsind)},\cV^{(\tsind)})=\textsc{PRK-HR}(\redb^{(\tsind-1)},\cX^{(\tsind-1)},\cV^{(\tsind-1)},\{\deimb{\idx}\}_{\idx=1}^\fedim,\{\deimi{\idx}\}_{\idx=1}^\fedim)$ as in \Cref{algo:PRK2-hr}
    \State $(\errind_{\star}^{(\tsind)},\merrappr_{\star}^{(\tsind)},\mdiscref_{\star}^{(\tsind)},\rostar^{(\tsind)})=\textsc{RA-EI}(\merrappr_{\star}^{(\tsind-1)},\mdiscref_{\star}^{(\tsind-1)},\rostar^{(\tsind-1)},\redb^{(\tsind)},\cX^{(\tsind)},\cV^{(\tsind)})$ as in \Cref{algo:RA-EI}
    \If{$\errind_{\star}^{(\tsind)}\geq C_1C_2^\mu\overline{\errind}_\star$}
        \State $(\redb^{(\tsind)},\cX^{(\tsind)},\cV^{(\tsind)})=\textsc{rank\_update}(\redb^{(\tsind)},\cX^{(\tsind)},\cV^{(\tsind)},\merrappr_\star^{(\tsind)},\gamma,\Psbfnt)$  as in \Cref{algo:rank-update}
        \State $\mu\gets\mu+1$
        \State $\overline{\errind}_\star\gets\errind^{(\tsind)}_\star$
    \EndIf
    \If{mod$(\tsind,\eimuf)=0$}
        \State $(\{\deimb{\idx}\}_{\idx=1}^\fedim,\{\deimi{\idx}\}_{\idx=1}^\fedim)=\textsc{greedyEIM}(\redb^{(\tsind)},\cX^{(\tsind)},\text{tol})$ as in \Cref{algo:gEIM}
    \EndIf
\EndFor
\EndProcedure}
\end{algorithmic}
\end{algorithm}

\section{Numerical experiments}\label{sec:num_exp}
In this section we test the proposed hyper-reduction strategy on two benchmark cases: the nonlinear Landau damping (NLLD) and the two-stream instability (TSI). The setup of the numerical experiments is analogous to \cite{HPR24}. The initial positions and velocities of the macro-particles are sampled from the perturbed distribution
\begin{equation}\label{eq:incond}
    f(0,x,v;\prm) = \left(1 + \prma\cos{(\wn x)}\right)f_v(v;\prmsd)
\end{equation}
where the parameter $\prm=(\prma,\prmsd)\in\prms\subset\Rbb^2$ is given by the amplitude $\prma$ of the perturbation and the standard deviation $\prmsd$ of the velocity distribution. We set $\Omega_x=\left[0,2\pi\wn^{-1}\right]$ and $\Omega_v=\left[-10,10\right]$. The value of the wavenumber $\wn$ and the expression of the velocity distribution $f_v$ will be specified for each test case. The initial condition is obtained by evaluating the inverse cumulative distribution function of $f$ at the points defined by the quasirandom Hammersley sequence. This choice is known to yield a significant noise reduction compared to random initialization \cite{Syd99}. The problem is solved for $p$ uniformly selected test parameters in $\prms$, and the time interval $[0,T]$ is discretized with a uniform time step $\Delta t$. Our goal is to assess the performance of the hyper-reduced system compared to the reduced and full order model in terms of numerical accuracy and computational efficiency. In particular, we consider the relative error in the Frobenius norm defined as
\begin{equation}\label{eq:rel_err}
    \solerr(t^{\tsind}) := \dfrac{\normF{\fos^{(\tsind)} - \Theta^{(\tsind)}}}{\normF{\fos^{(\tsind)}}}
\end{equation}
and we distinguish between $\solerr^{r}$ and $\solerr^{\hr}$ depending on whether $\Theta^{(\tsind)}$ is the solution of the reduced or hyper-reduced model at time $t^\tsind$, respectively.
In the rank adaptive case, we are interested in the average relative error in time, that we define as
\begin{equation}\label{eq:rel_err_avg}
    \avgsolerr := \dfrac{1}{T}\int_0^T\frac{\normF{\fos(t) - \Theta(t)}}{\normF{\fos(t)}}\,dt.
\end{equation}
In practice, the integral appearing in \eqref{eq:rel_err_avg} is approximated using a suitable quadrature rule.
We also measure the conservation of the full order Hamiltonian by means of the quantity
\begin{equation}\label{eq:err_ham}
    \hamerr(t^\tsind) = \dfrac{1}{p}\sum_{s=1}^p\frac{\lvert\ham(\Theta^{(\tsind)}(\prm_s))-\ham(\Theta^{(0)}(\prm_s))\rvert}{\lvert\ham(\Theta^{(0)}(\prm_s))\rvert},
\end{equation}
where $\Theta^{(\tsind)}$ is either the solution of the FOM, the solution of the ROM or the solution of the hROM.
This is an indicator of the relative variation of the full order Hamiltonian $\ham$ at time $t^\tsind$ with respect to the initial condition, averaged over all test parameters. We denote by $\hamerr^{\fom}$, $\hamerr^{\rom}$ and $\hamerr^{\hr}$ the values of $\hamerr$ associated to the solution of the FOM, ROM and hROM, respectively. 

In all numerical experiments on the hROM, the EIM basis and interpolation indices are reconstructed every $\eimuf=20$ time steps with tolerance $10^{-4}$ in \Cref{algo:gEIM}.


\subsection{Nonlinear Landau damping}\label{sec:NLLD}
In the first test case we consider the nonlinear Landau damping benchmark. We take the initial condition \eqref{eq:incond} with velocity distribution
\begin{equation*}
    f_v(v;\prmsd) = \frac{1}{\sqrt{2\pi\prmsd^2}}\text{exp}\left(-\frac{v^2}{2\prmsd^2}\right).
\end{equation*}
We set $\prms=[0.46,0.5]\times[0.96,1]$ and $\wn=0.5$. The number of macro-particles is set to $N=10^5$ and the spatial domain $\Omega_x$ is discretized using $N_x=64$ uniform spatial intervals. The problem is solved for $p=100$ test parameters until the final time $T=40$ on a uniform time grid of $N_t=20000$ intervals, corresponding to a time step $\Delta t=0.002$. The choice of a small time step allows to study the approximation error associated to the ROM and hROM without spurious contributions due to the temporal discretization. 

We first focus on the non-rank-adaptive case (NRA) with $n=3$, and we show in \Cref{fig:NLLD_distfun} the distribution function for one randomly chosen test parameter, $\prm=(\prma,\prmsd)=(0.4644, 0.9867)$, obtained with the FOM, ROM and hROM at three time instants, $t\in\{0,20,40\}$. 

\begin{figure}[H]
    \centering
    \includegraphics[width=0.85\textwidth]{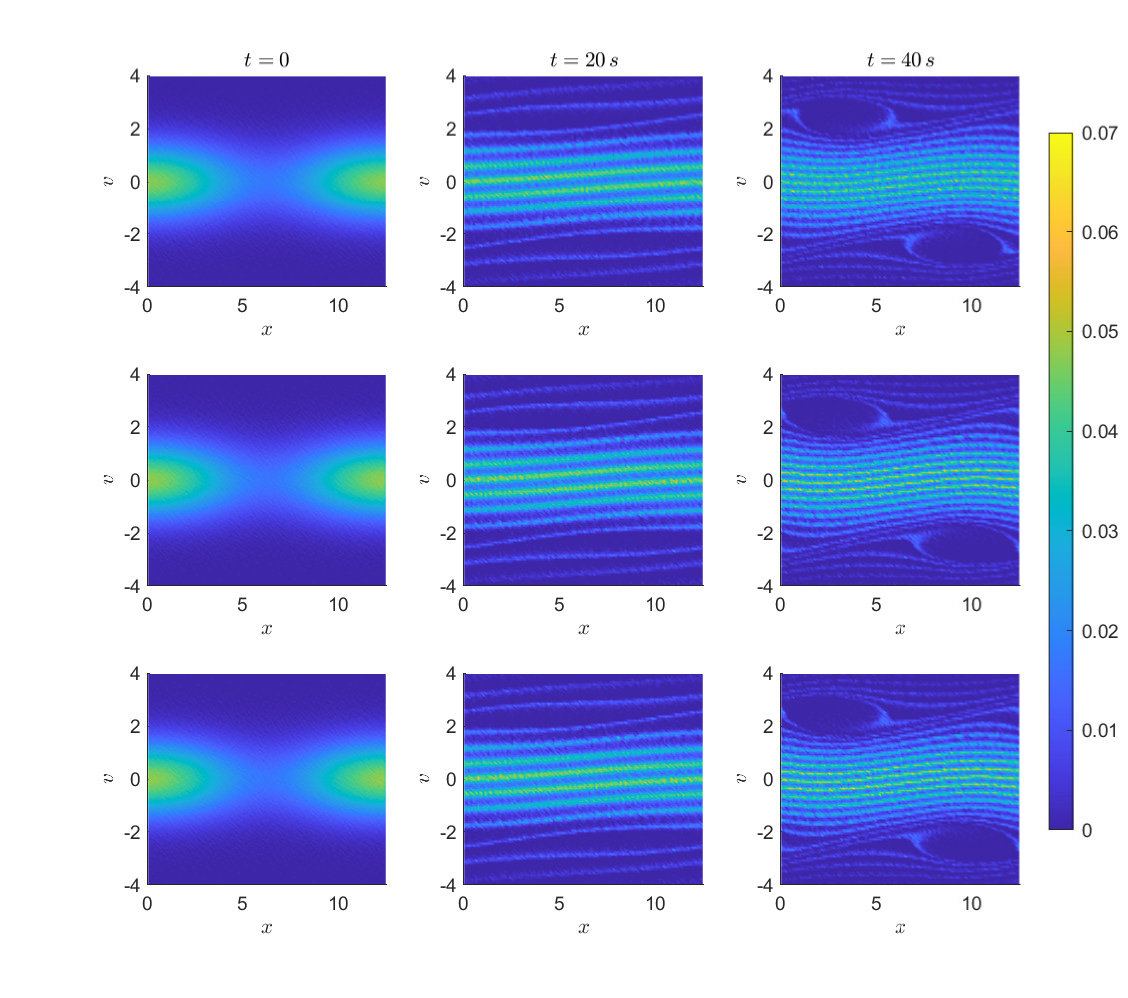}
    \vspace{-3em}
    \caption{\footnotesize NLLD. Numerical distribution function $f_h(t,x,v;\prm)$ at times $t=0$ (left), $t=20$ (center) and $t=40$ (right) for $\prm=(\prma,\prmsd)=(0.4644, 0.9867)$. Comparison between the FOM (first row), the ROM with $n=3$ (second row), and the hROM with $n=3$ and $m$ varying (third row).}
    \label{fig:NLLD_distfun}
\end{figure}

\noindent We observe that both the reduced and the hyper-reduced models correctly reproduce the qualitative behavior of the full order solution.
The Landau damping leads to a decrease of the potential energy in the first part of the simulation. As the dynamics evolves, the Landau damping rate decreases and trapped particles cause the potential energy of the system to increase (see also \Cref{fig:NLLD_energy}).

For a quantitative assessment, we report in \Cref{fig:NLLD_errors} the evolution of the relative errors over time. Moreover, the computational runtimes are reported in \Cref{tab:NLLD_ct} together with the relative errors at the final time. The results show that the proposed hyper-reduction strategy with $n=2$ and $n=3$ yields a reduction of the computational time by a factor of $13$ and $11$ with respect to the full order model, respectively. By contrast, solving the reduced order model does not yield any computational speed up. Moreover, the accuracy of the hROM is comparable to that of the ROM at all times.

\begin{figure}[H]
\centering
\includegraphics[]{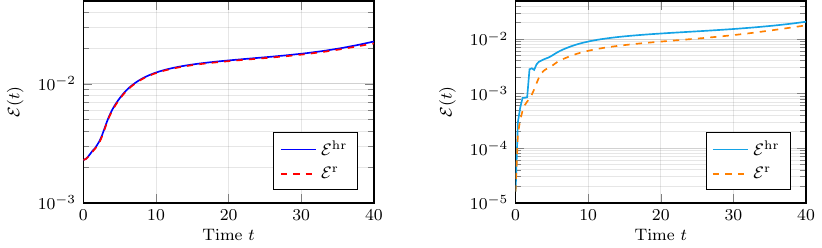}
\caption{\footnotesize NLLD. Relative errors \eqref{eq:rel_err} of the ROM and hROM with respect to the FOM solution: $n=2$ (left) and $n=3$ (right).}\label{fig:NLLD_errors}
\end{figure}

\renewcommand{\arraystretch}{1.25}
\begin{table}[H]
\small
    \vspace{0.5cm}
    \centering
    \begin{tabular}{c|c|c|c}
    $n=2$ & Runtime & $\solerr(T)$ & speedup\\
    \hline
    ROM & 15521.9\,s & 2.19e-02 & 0.99 \\ \hline
    hROM & 1137.7\,s & 2.27e-02 & 13.48 \\
    \end{tabular}
    \qquad\qquad
    \begin{tabular}{c|c|c|c}
    $n=3$ & Runtime & $\solerr(T)$ & speedup \\
    \hline
    ROM & 15728.3\,s & 1.77e-02 & 0.98 \\ \hline
    hROM & 1349.7\,s & 2.07e-02 & 11.36 \\ 
    \end{tabular}
    \caption{\footnotesize NLLD. Computational runtimes and relative errors at the final time of the ROM and hROM for $n=2$ and $n=3$. We also report the reduction factor, defined as the ratio between the runtime of the FOM and the runtime of the ROM or hROM. The full order model has dimension $2N=2\times10^5$, and it is solved in $15339.1\,s$. The number of test parameters is $p=100$.}\label{tab:NLLD_ct}
\end{table}

In \Cref{fig:NLLD_m} we show the evolution of the dimension $m$ of the EIM space, or equivalently the number of macro-particles selected by the greedy \Cref{algo:gEIM}. We recall that in the hROM the particle-to-grid map is only evaluated at $m$ particles: in this test case, this number only constitutes about $0.4\%$ of the total number of particles if $n=2$, and $1\%$ if $n=3$. The larger value of $m$ in the case $n=3$ can be ascribed to the fact that the reducibility properties of the gradient of the reduced Hamiltonian degrade as the dimension of the reduced space increases, so that a largest EIM space is required to achieve the same level of accuracy \cite{PV23}. This is also the reason behind the slight decrease of the runtime gain factor observed for $n=3$ in \Cref{tab:NLLD_ct}. We also mention that it might be possible to optimize the value of $m$ at each update by developing an adaptive strategy for the selection of the stopping tolerance of \Cref{algo:gEIM}. While its value is fixed in these numerical experiments, investigating the effect of a dynamically adapted tolerance on the overall efficiency would be an interesting direction for future work.
\begin{figure}[H]
\centering
\includegraphics[]{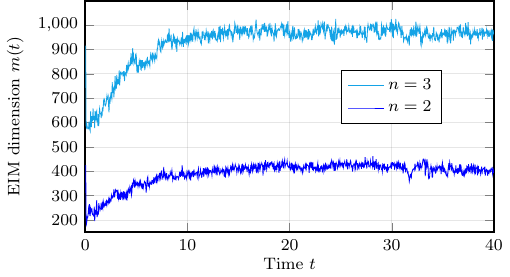}
\caption{\footnotesize NLLD. Evolution of the dimension of the EIM approximation space.}
\label{fig:NLLD_m}
\end{figure}

In \Cref{fig:NLLD_energy} we plot the evolution of the electric potential energy $\nlham$ \eqref{eq:Hamiltonian} evaluated at the
FOM, ROM and hROM solutions for two random values of the test parameter $\prm$. We remark that, although random choices of $\prm$ are considered here for illustration purposes, we observed qualitatively similar results for different values of the test parameter.
The error \eqref{eq:err_ham} in the conservation of the Hamiltonian is plotted in \Cref{fig:NLLD_hamiltonian}. It can be observed that the full order Hamiltonian is not preserved exactly by the reduced and hyper-reduced order models. As noted in \cite{PV25}, the reason for this behavior can be attributed to several factors. First, the time integrator employed to evolve the expansion coefficients is symplectic but does not preserve the Hamiltonian at each time step. Second, the reduced Hamiltonian \eqref{eq:red_Ham} and the hyper-reduced Hamiltonian \eqref{eq:hred_Ham} are approximations of the full order Hamiltonian \eqref{eq:Hamiltonian}. Third, both the reduced basis and the EIM basis are updated in the hROM, introducing a further error in the conservation of invariants. Nevertheless, the error in the conservation of the Hamiltonian remains bounded in the hROM, and its time evolution is qualitatively similar to the benchmark provided by the ROM.
\begin{figure}[H]
\centering
\includegraphics[]{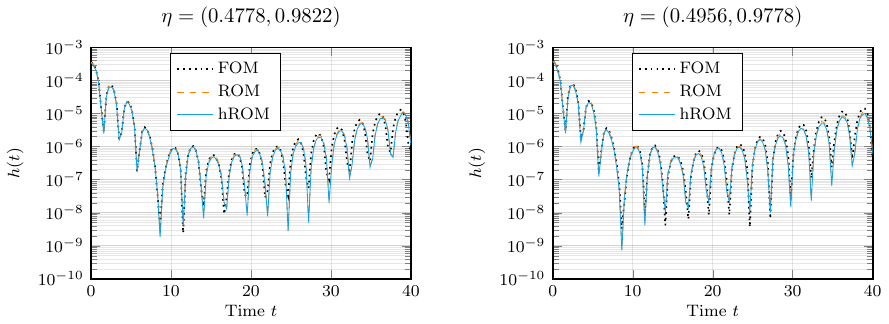}
\caption{\footnotesize NLLD. Evolution of the electric energy $\nlham$ evaluated at the FOM, ROM, and hROM solutions with $n=3$ and for two random choices of the parameter $\prm$.}\label{fig:NLLD_energy}
\end{figure}

\begin{figure}[H]
\centering
\includegraphics[]{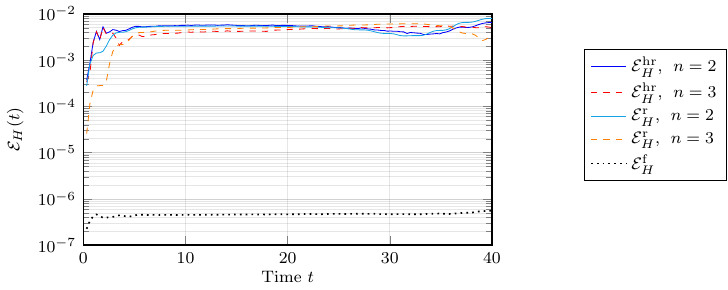}
\caption{\footnotesize NLLD. Evolution of the error in the conservation of the Hamiltonian \eqref{eq:err_ham}. Comparison between the FOM, the ROM, and the hROM with $n=2$ and $n=3$.}\label{fig:NLLD_hamiltonian}
\end{figure}

Finally, we assess the efficiency of the hyper-reduced model as a function of the number of test parameters when compared to the full order model and the reduced order model, for a fixed dimension of the reduced space. To this end, we set $n=3$ and we compute the average runtime obtained in the first $100$ time steps. Results are shown in \Cref{fig:NLLD_comptimes}. As expected, the ROM is as computationally expensive as the FOM, and its runtime is proportional to the number of test parameters. On the other hand, the computational complexity of the hROM grows less rapidly for smaller values of $p$. This is due to the fact that, in this regime, the most expensive operations involved in the solution of the hyper-reduced system are those required to evolve the reduced basis, whose computational complexity is $O(Nn^2)$. We remark that solving the FOM or ROM for $p=10^2$ is approximately as demanding as solving the hROM with $p=5\cdot 10^3$, and the total runtime is reduced by a factor of $50$ when $p=10^4$. 
\begin{figure}[H]
\centering
\includegraphics[]{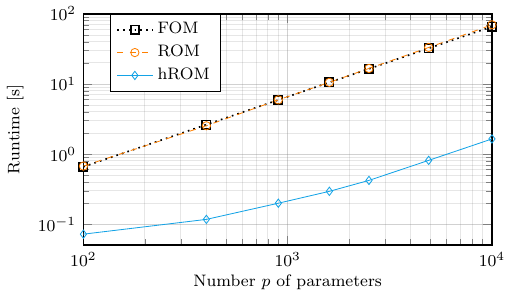}
\caption{\footnotesize NLLD. Average runtime per time step as a function of the number $p$ of test parameters.}\label{fig:NLLD_comptimes}
\end{figure}

\subsubsection{Rank adaptivity}
In order to curb the growth of the numerical errors observed in \Cref{fig:NLLD_errors}, we apply the rank-adaptive (RA) strategy proposed in \Cref{sec:rank-adaptive}. We focus on the hROM, and we set $N=10^5$ and $p=100$ as in the previous section. In order to measure the variability across the parameter space, we study the evolution of the numerical rank of the full order solution $\fos(t)\in\Rbb^{2N\times p}$ at each time $t$ in terms of its $\varepsilon$-rank, defined as
\begin{equation*}
    \text{rank}_\varepsilon(\fos(t)) := \min\left\{n\in\mathbb{N} : \frac{\normF{\fos(t)-\fosvd{n}(t)}}{\normF{\fos(t)}}<\varepsilon \quad \text{ where $\fosvd{n}(t)$ is the $n$-truncated SVD of $\fos(t)$}\right\}.
\end{equation*}
This quantity is shown for different values of $\varepsilon$ in \Cref{fig:NLLD_numrank}. The rapid growth of the numerical rank motivates the need to adapt the dimension of the reduced space.
\begin{figure}[H]
\centering
\includegraphics[]{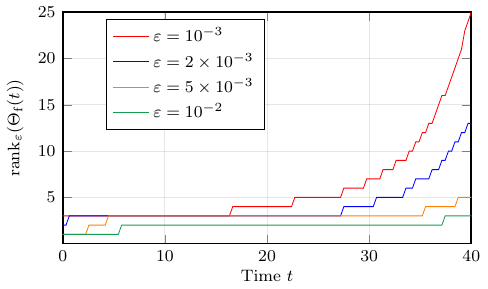}
\caption{\footnotesize NLLD. Evolution of the $\varepsilon$-rank of the full order solution $\fos(t)$ for different values of $\varepsilon$.}\label{fig:NLLD_numrank}
\end{figure}

In a first test, we compute the error indicator at every time step as in \Cref{algo:RA-EI}, with the approximate residual evaluated at $\spei=p$ sample parameters. The basis update is performed according to criterion \eqref{eq:update_criterion} with $\cone=\ctwo=1.05$. At each update, one new vector is added to the reduced basis, and the coefficient matrices of the reduced solution are augmented with two rows of zeros, that is, we set $\gamma=0$ in \Cref{algo:rank-update}. We recall that this is the procedure proposed in \cite{HPR24}, and it is equivalent to projecting the reduced solution before the update onto the enlarged reduced space. We report in \Cref{fig:NLLD_RA_updzeros} the evolution of the relative error and of the reduced dimension $n(t)$ in the rank-adaptive case, and a comparison with the non-rank-adaptive case with $n=2$. 
\begin{figure}[H]
\centering
\includegraphics[]{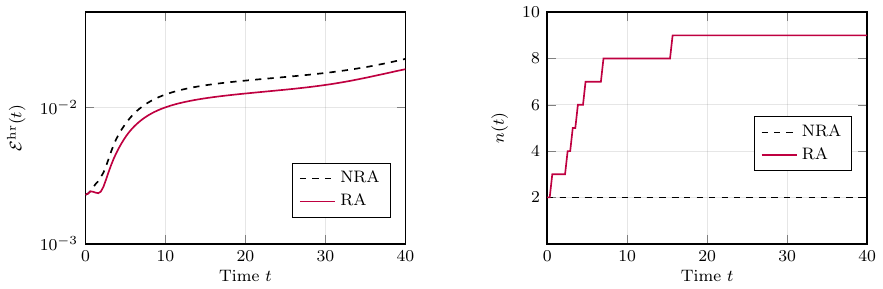}
\caption{\footnotesize NLLD. 
Evolution of the relative error \eqref{eq:rel_err} (left) and of the dimension of the reduced basis (right) for the rank-adaptive and non-rank-adaptive algorithms.
Rank-adaptive system with $\gamma=0$ in \Cref{algo:rank-update} and $\cone=\ctwo=1.05$ in \eqref{eq:update_criterion}. The error indicator is computed at all time steps using all test parameters.}\label{fig:NLLD_RA_updzeros}
\end{figure}

We observe two major drawbacks associated to this rank-adaptive strategy. First, increasing the dimension of the reduced space does not result in a significant reduction of the growth of the numerical error. Second, the computational time required to solve the rank-adaptive hyper-reduced model exceeds the time required to solve the full order model, as shown in the first two rows of \Cref{tab:NLLD_RA_updnozeros_p}: this is due to the fact that, at each time step, the full order electric field is evaluated at all parameters to compute the error indicator. We address these two issues separately.

First, we consider a different strategy for the update of the coefficient matrices of the reduced solution by setting $\gamma=1$ in \Cref{algo:rank-update}. As in the previous test, the error indicator is computed at all time steps based on all test parameters. We also conduct a sensitivity analysis on the hyper-parameters appearing in the criterion \eqref{eq:update_criterion} for rank update by choosing four combinations of the constants $\cone$ and $\ctwo$. We present the evolution of the numerical errors and of the reduced basis dimension in \Cref{fig:NLLD_RA_updnozeros_p}. We also report the computational times and the average relative errors \eqref{eq:rel_err_avg} in \Cref{tab:NLLD_RA_updnozeros_p}.

\begin{figure}[H]
\centering
\includegraphics[]{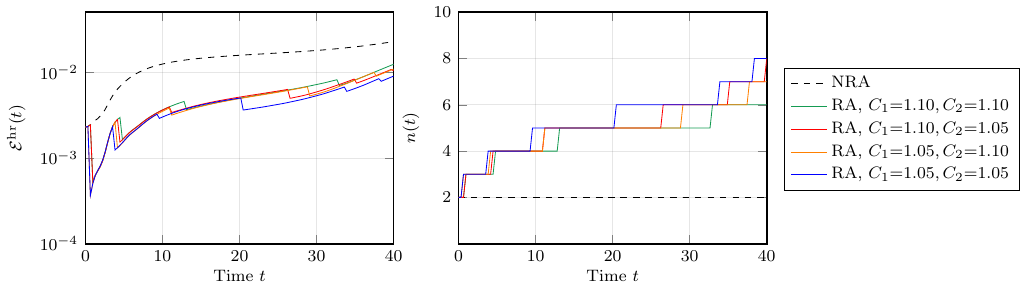}
\caption{\footnotesize NLLD. Evolution of the relative error \eqref{eq:rel_err} (left) and of the dimension of the reduced basis (right)
for the rank-adaptive and non-rank-adaptive algorithms.
Rank-adaptive system with $\gamma=1$ in \Cref{algo:rank-update} and different choices of $\cone$ and $\ctwo$ in \eqref{eq:update_criterion}. The error indicator is computed at all time steps using all test parameters.}\label{fig:NLLD_RA_updnozeros_p}
\end{figure}

\begin{table}[H]
\small
\begin{center}
\begin{tabular}{c|c|c}
    Model & Runtime  & $\avgsolerr$ \\ \hline
    NRA & 1137.7\,s & 1.46e-02 \\ \hline
    RA, $\gamma=0$, $\cone=1.05$, $\ctwo=1.05$ & 18086.9\,s & 1.19e-02 \\ \hline
    RA, $\gamma=1$, $\cone=1.10$, $\ctwo=1.10$ & 15968.4\,s & 5.33e-03 \\
    RA, $\gamma=1$, $\cone=1.10$, $\ctwo=1.05$ & 16188.0\,s & 5.00e-03 \\
    RA, $\gamma=1$, $\cone=1.05$, $\ctwo=1.10$ & 16117.7\,s & 4.95e-03 \\
    RA, $\gamma=1$, $\cone=1.05$, $\ctwo=1.05$  & 16354.8\,s & 4.31e-03
\end{tabular}\caption{\footnotesize NLLD. Computational time and average relative error in time \eqref{eq:rel_err_avg} for different rank-adaptive strategies, and comparison with the non-rank-adaptive case. The error indicator for rank adaptivity is computed at all times based on all $p$ test parameters. The FOM is solved in $15339.1\,s$.}\label{tab:NLLD_RA_updnozeros_p}
\end{center}
\end{table}

We remark that, when $\gamma=1$, the average relative error \eqref{eq:rel_err_avg} is around $5\cdot10^{-3}$, about a third of the value attained in the non-rank-adaptive model, and less than half of the error obtained with $\gamma=0$. We also observe that the constants $\cone$ and $\ctwo$ have an impact on the algorithm performances: by reducing their values, it is possible to increase the frequencies of the rank updates, which results in smaller errors at the price of a slightly higher computational time. While this test shows the benefit of initializing the rows of the coefficient matrices associated to new modes to some nonzero values, the computational cost of this strategy is still prohibitive, because the error indicator is computed at all times using all test parameters. We address this issue by means of the interpolation strategy outlined in \Cref{sec:rank-adaptive}. In this work, we consider second order Legendre polynomials as basis functions: we set $\nti=6$ in \eqref{eq:interp} and
\begin{align*}
    &\sbfnt_1(\prm) = 1, \quad \sbfnt_2(\prm) = \frac{\prma-\overline{\prma}}{\Delta\prma}, \quad \sbfnt_3(\prm) = \frac{\prmsd-\overline{\prmsd}}{\Delta\prmsd}, \quad \sbfnt_4(\prm) = \frac{1}{2}\left[3\left(\frac{\prma-\overline{\prma}}{\Delta\prma}\right)^2-1\right], \\& \sbfnt_5(\prm) = \frac{1}{2}\left[3\left(\frac{\prmsd-\overline{\prmsd}}{\Delta\prmsd}\right)^2-1\right], \quad \sbfnt_6(\prm) = \frac{\prma-\overline{\prma}}{\Delta\prma}\frac{\prmsd-\overline{\prmsd}}{\Delta\prmsd},
\end{align*}
where $\prm=(\prma,\prmsd)\in[\prma_L,\prma_R]\times[\prmsd_L,\prmsd_R]=\prms$ and
\begin{equation*}
    \overline{\prma} = \frac{\prma_L+\prma_R}{2}, \quad \overline{\prmsd} = \frac{\prmsd_L+\prmsd_R}{2}, \quad \Delta\prma = \frac{\prma_R-\prma_L}{2}, \quad \Delta\prmsd = \frac{\prmsd_R-\prmsd_L}{2}.
\end{equation*}
Then, we randomly select $\spei=\nti=6\ll p$ sample parameters that are used to compute the error indicator as in \Cref{algo:RA-EI} at each time step. Both the set of sample parameters and the basis functions $\sbfnt_i$, $i=1,\dots,\nti$, are fixed throughout the simulation. We test our strategy by solving the rank-adaptive scheme with $\gamma=1$ and $\cone=\ctwo=1.05$. We report the evolution of the numerical error and of the dimension of the reduced basis in \Cref{fig:NLLD_RA_updnozeros_pstar}.

\begin{figure}[H]
\centering
\includegraphics[]{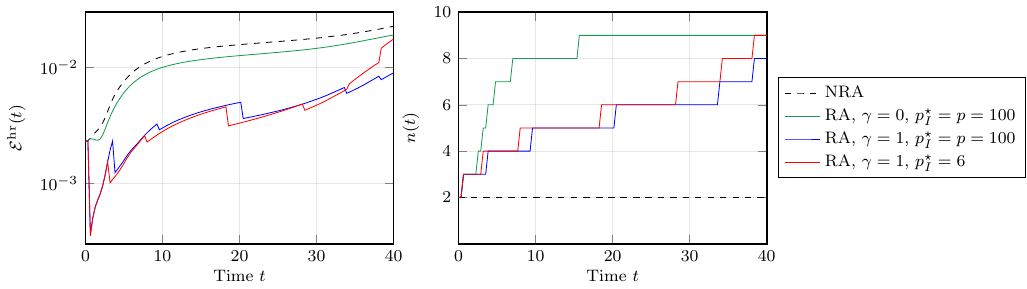}
\caption{\footnotesize NLLD.
Evolution of the relative error \eqref{eq:rel_err} (left) and of the dimension of the reduced basis (right)
for the rank-adaptive and non-rank-adaptive algorithms.
Rank-adaptive system with $\cone=\ctwo=1.05$ in \eqref{eq:update_criterion}.}\label{fig:NLLD_RA_updnozeros_pstar}
\end{figure}

\begin{table}[H]
\small
\begin{center}
\begin{tabular}{c|c|c}
    Model & Runtime  & $\avgsolerr$ \\ \hline
    NRA & 1137.7\,s & 1.46e-02 \\ \hline
    RA, $\gamma=0$, $\spei=p=100$  & 18086.9\,s & 1.19e-02 \\ \hline
    RA, $\gamma=1$, $\spei=p=100$  & 16354.8\,s & 4.31e-03 \\ \hline
    RA, $\gamma=1$, $\spei=6$ & 3300.6\,s & 4.61e-03
\end{tabular}\caption{\footnotesize NLLD. Computational time and average relative error in time \eqref{eq:rel_err_avg} for different rank-adaptive strategies, and comparison with the non-rank-adaptive case. In the rank-adaptive models, the parameters in the update criterion are $\cone=\ctwo=1.05$. The FOM is solved in $15339.1\,s$.}\label{tab:NLLD_RA_updnozeros_pstar}
\end{center}
\end{table}

The computational runtimes are reported in \Cref{tab:NLLD_RA_updnozeros_pstar}.
Since the full order electric field is only evaluated at a small number of sample parameters, the total computational time is drastically reduced. On the other hand, the average error in time $\avgsolerr$ is comparable to the one obtained with $\spei=p$. The total runtime can be reduced further, for example, by computing the error indicator every $\widehat{\delta}>1$ time steps. However, we observe that the accuracy of the hyper-reduced solution degrades slightly in the final stages of the simulation, as seen in \Cref{fig:NLLD_RA_updnozeros_pstar} (left). This can be addressed by either modifying the set of sample parameters or by considering a different set of basis functions for interpolation. The development of such strategies is left for future work.

Finally, we show in \Cref{fig:NLLD_RA_ham} the evolution of the Hamiltonian error \eqref{eq:err_ham}. Setting $\gamma=0$ in the rank-adaptive algorithm ensures better conservation properties, since the reduced solutions before and after each rank update coincide, as mentioned in \Cref{sec:coeff_upd}. On the other hand, the conservation error $\hamerr$ remains bounded when $\gamma=1$, and the results obtained with $\spei=6$ is comparable to those achieved with $\spei=p$.
\begin{figure}[H]
\centering
\includegraphics[]{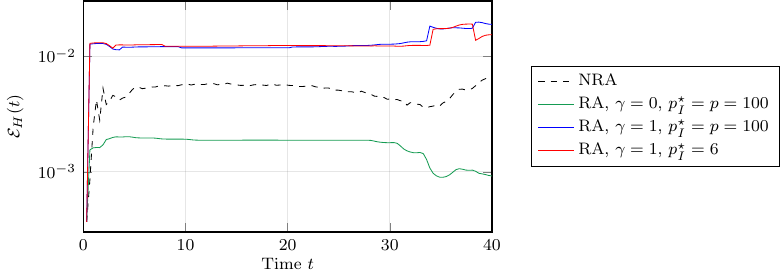}
\caption{\footnotesize NLLD. Evolution of the error in the Hamiltonian conservation \eqref{eq:err_ham}. Rank-adaptive system with $\cone=\ctwo=1.05$. }\label{fig:NLLD_RA_ham}
\end{figure}

\subsection{Two-stream instability}
As a second test case, we consider the two-stream instability (TSI) benchmark. For this test, the initial condition is as in \eqref{eq:incond} with
\begin{equation*}
    f_v(v;\prmsd) = \frac{1}{2\sqrt{2\pi\prmsd^2}}\text{exp}\left(-\frac{(v-3)^2}{2\prmsd^2}\right) + \frac{1}{2\sqrt{2\pi\prmsd^2}}\text{exp}\left(-\frac{(v+3)^2}{2\prmsd^2}\right).
\end{equation*}
Here the parameter $\prm=(\prma,\prmsd)$ varies in the parameter space $\prms=[0.009,0.011]\times[0.98,1.02]$ and the wavenumber is $\wn=0.2$. For spatial discretization, we choose $N=1.5\times10^5$ macro-particles and $N_x=64$ mesh intervals. The final time is $T=20$, and the time interval is discretized using $N_t=8000$ time steps, corresponding to $\Delta t = 0.0025$.

\begin{figure}[H]
    \centering
    \includegraphics[width=0.85\textwidth]{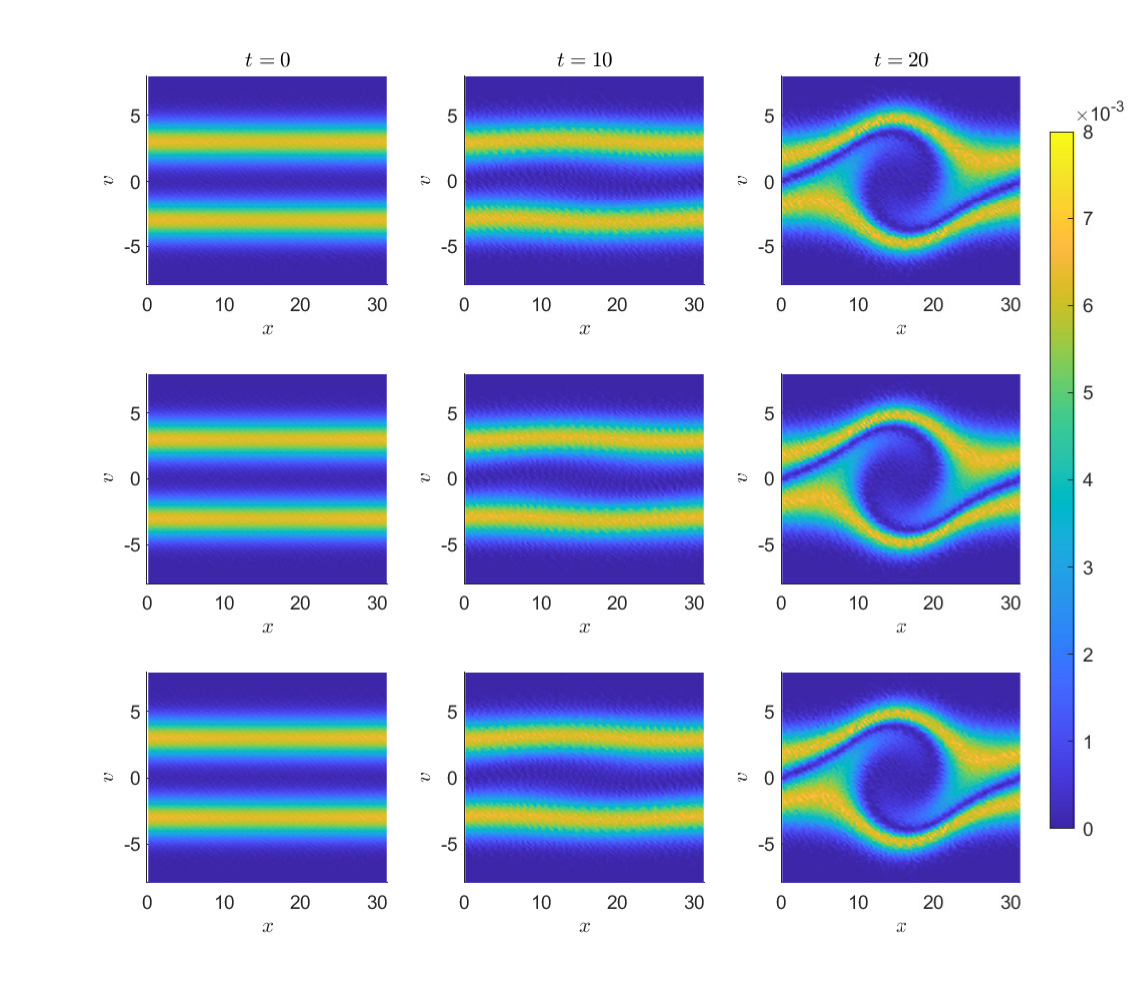}
    \vspace{-3em}
    \caption{\footnotesize TSI. Numerical distribution function $f_h(t,x,v;\prm)$ at times $t=0$ (left), $t=10$ (center) and $t=20$ (right) for $\prm=(\prma,\prmsd)=(0.0092, 1.0067)$. Comparison between the FOM (first row), the ROM with $n=3$ (second row), and the hROM with $n=3$ and $m$ varying (third row).}
    \label{fig:TSI_distfun}
\end{figure}

\Cref{fig:TSI_distfun} shows the distribution function at three time instants for one choice of the test parameter. The problem is characterized by an instability generated by two streams of charged particles moving in opposite directions transferring energy to the plasma wave. The numerical solution obtained with the non-rank-adaptive ROM and hROM with $n=3$ are in agreement with the solution obtained with the FOM.

We show in \Cref{fig:TSI_errors} the evolution of the relative error of the non-rank-adaptive ROM and hROM.
The corresponding computational runtimes are reported in \Cref{tab:TSI_ct}.
We observe a reduction of the computational time by a factor of $15$ and $13$ in the hROM with $n=2$ and $n=3$, respectively, with respect to the FOM.

\begin{figure}[H]
\centering
\includegraphics[]{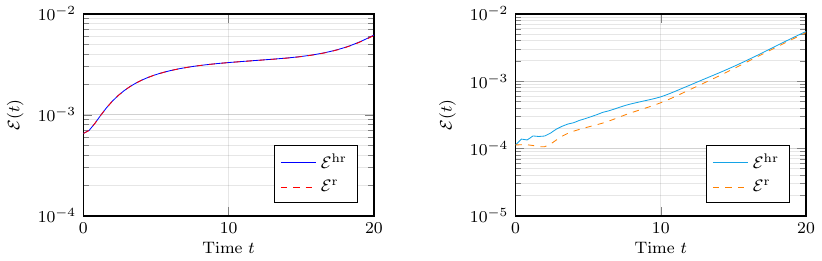}
\caption{\footnotesize TSI. Relative errors \eqref{eq:rel_err} of the ROM and hROM with respect to the FOM solution: $n=2$ (left) and $n=3$ (right).}\label{fig:TSI_errors}
\end{figure}

\begin{table}[H]
\small
    \vspace{0.5cm}
    \centering
    \begin{tabular}{c|c|c|c}
    $n=2$ & Runtime & $\solerr(T)$ & speedup \\
    \hline
    ROM & 10849.2\,s & 6.10e-03 & 0.96 \\ \hline
    hROM & 708.5\,s & 6.16e-03 & 14.67 \\ 
    \end{tabular}
    \qquad\qquad
    \begin{tabular}{c|c|c|c}
    $n=3$ & Runtime & $\solerr(T)$ & speedup \\
    \hline
    ROM & 11175.0\,s & 5.26e-03 & 0.93 \\ \hline
    hROM & 813.7\,s & 5.47e-03 & 12.77 \\
    \end{tabular}
    \caption{\footnotesize TSI. Computational runtimes and relative errors at the final time of the ROM and hROM for $n=2$ and $n=3$. We also report the reduction factor, defined as the ratio between the computational times in the FOM and the ROM or hROM. The full order model has dimension $2N=3\times10^5$, and it is solved in $10391.8\,s$. The number of test parameters is $p=100$.}\label{tab:TSI_ct}
\end{table}

The evolution of the dimension $m$ of the EIM space is shown in \Cref{fig:TSI_m}. The particle-to-grid map is only evaluated at most at $0.2\%$ and $0.4\%$ of the total number of particles for $n=2$ and $n=3$, respectively. We observe that these ratios are slightly lower than those obtained in \Cref{sec:NLLD}, yielding higher values of the reduction factors, as seen in \Cref{tab:TSI_ct}. We also remark that the growth of the EIM dimension in the two cases resembles the evolution of the numerical errors shown in \Cref{fig:TSI_errors}, implying the existence of a relationship between the reducibility of the solution and the reducibility of the nonlinear term.  
\begin{figure}[H]
\centering
\includegraphics[]{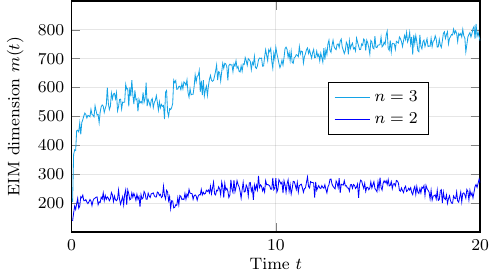}
\caption{\footnotesize TSI. Evolution of the dimension of the EIM approximation space.}\label{fig:TSI_m}
\end{figure}

Next, \Cref{fig:TSI_energy} depicts the electric energy in the ROM and hROM, and a comparison with the same quantity in the FOM for two randomly selected instances of the test parameter. The reduced models provide a good approximation at all times. Qualitatively similar results were observed for different values of $\prm$. Analogously, the average variation of the full order Hamiltonian evaluated at the hyper-reduced solution remains bounded in the time interval $[0,20]$, as shown in \Cref{fig:TSI_hamiltonian}.

\begin{figure}[H]
\centering
\includegraphics[]{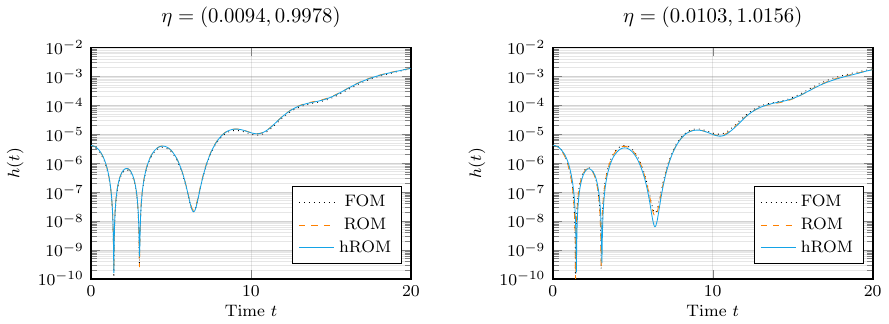}
\caption{\footnotesize TSI. Electric energy $\nlham(X(t,\prm))$ in the FOM, ROM, and hROM with $n=3$ for two random choices of the parameter $\prm$.}\label{fig:TSI_energy}
\end{figure}


\begin{figure}[H]
\centering
\includegraphics[]{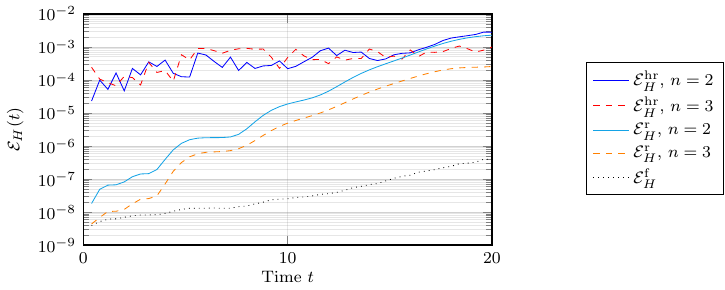}
\caption{\footnotesize TSI. Evolution of the error in the conservation of the Hamiltonian \eqref{eq:err_ham}. Comparison between the FOM, the ROM, and the hROM with $n=2$ and $n=3$.}\label{fig:TSI_hamiltonian}
\end{figure}

We conclude this section with the analysis of the computational cost of the hROM as a function of the number of test parameters. Results obtained with $n=3$ and $p$ ranging from $10^2$ to $10^4$ are reported in \Cref{fig:TSI_comptimes}, where the average runtime obtained in the first 100 time steps is shown. We observe that solving the hROM for $p=10^4$ test parameters is roughly as expensive as solving the ROM and FOM for $p=10^2$ test parameters, and around $100$ times cheaper than solving the FOM for the same value of $p$. 
\begin{figure}[H]
\centering
\includegraphics[]{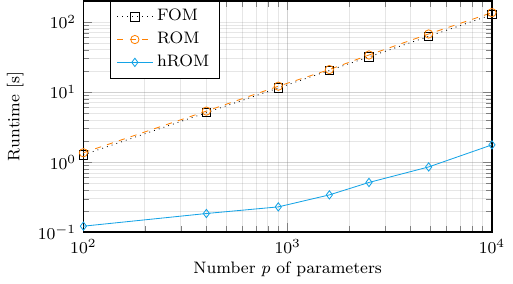}
\caption{\footnotesize TSI. Average runtime per time step as a function of the number $p$ of test parameters.}\label{fig:TSI_comptimes}
\end{figure}

\subsubsection{Rank adaptivity}
In the next experiment, we run the simulation for a longer time by setting $T=30$. As shown in \Cref{fig:TSI_numrank}, the numerical rank of the full order solution exhibits a moderate growth for $t>20$. This suggests that, in this new scenario, the accuracy of the non-rank-adaptive strategy might degrade over time, and it might be necessary to adapt the dimension of the reduced space. To address this problem, we numerically assess the performance of the rank-adaptive strategy presented in \Cref{sec:rank-adaptive}.

\begin{figure}[H]
\centering
\includegraphics[]{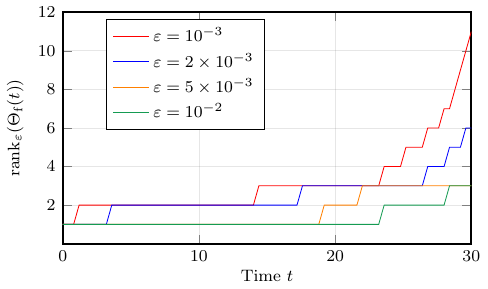}
\caption{\footnotesize TSI. Evolution of the $\varepsilon$-rank of the full order solution $\fos(t)$ for different values of $\varepsilon$.}\label{fig:TSI_numrank}
\end{figure}
First, we consider the case where the coefficient matrices of the reduced solution are augmented with rows of zeros at each rank adaptation, that is, we set $\gamma=0$ in \Cref{algo:rank-update}. Moreover, we compute the error indicator as in \Cref{algo:RA-EI} with $\spei=p$, and we set $\cone=\ctwo=1.05$ in the criterion \eqref{eq:update_criterion}. Results are shown in \Cref{fig:TSI_RA_updzeros}. The relative error in the hROM with respect to the full order solution is not significantly lower than in the non-rank-adaptive case (left plot) despite the increase of the reduced space dimension (right plot).
\begin{figure}[H]
\centering
\includegraphics[]{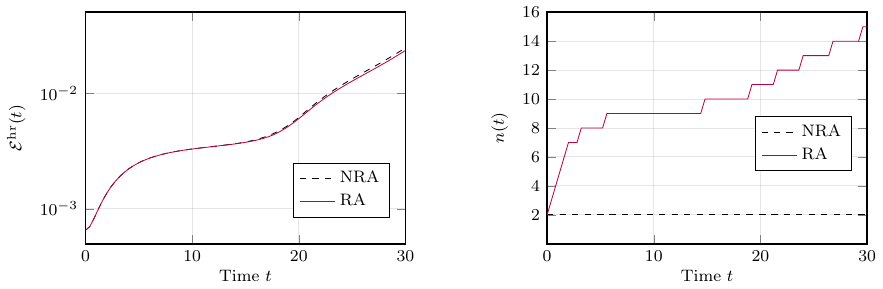}
\caption{\footnotesize TSI. Evolution of the relative error \eqref{eq:rel_err} (left) and of the dimension of the reduced basis (right) for the rank-adaptive and non-rank-adaptive algorithms.
Rank-adaptive system with $\gamma=0$ in \Cref{algo:rank-update} and $\cone=\ctwo=1.05$ in \eqref{eq:update_criterion}. The error indicator is computed at all time steps using all test parameters.}\label{fig:TSI_RA_updzeros}
\end{figure}

By setting $\gamma=1$ in \Cref{algo:rank-update}, it is possible to control the growth of the numerical errors, as reported in \Cref{fig:TSI_RA_updnozeros_p} for different choices of the parameters $\cone$ and $\ctwo$ in the criterion \eqref{eq:update_criterion}. Average relative errors in the time interval $[0,T]$ are shown in \Cref{tab:TSI_RA_updnozeros_p}: we observe a reduction by a factor of $5$ compared to the non-rank-adaptive case, and to the rank-adaptive case with $\gamma=0$. Nevertheless, the computational times are comparable to the FOM, because the error indicator is computed at all times based on all $p$ test parameters in these simulations.
\begin{figure}[H]
\centering
\includegraphics[]{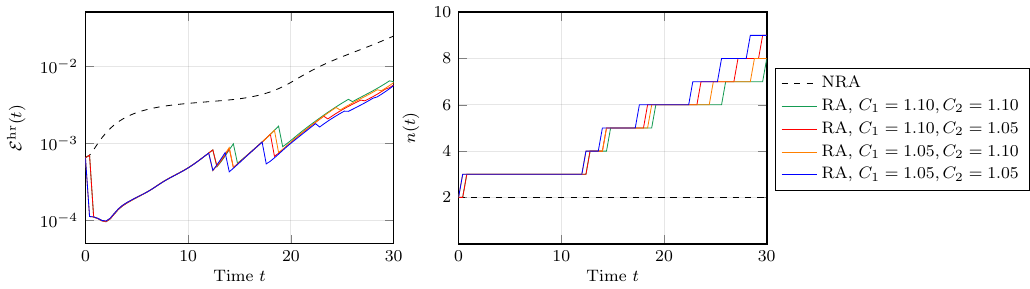}
\caption{\footnotesize TSI. Evolution of the relative error \eqref{eq:rel_err} (left) and of the dimension of the reduced basis (right) for the rank-adaptive and non-rank-adaptive algorithms.
Rank-adaptive system with $\gamma=1$ in \Cref{algo:rank-update} and different choices of $\cone$ and $\ctwo$ in \eqref{eq:update_criterion}. The error indicator is computed at all time steps using all test parameters.}\label{fig:TSI_RA_updnozeros_p}
\end{figure}

\begin{table}[H]
\small
\begin{center}
\begin{tabular}{c|c|c}
    Model & Runtime & $\avgsolerr$ \\ \hline
    NRA & 1115.1 & 6.82e-03 \\ \hline
    RA, $\gamma=0$, $\cone=1.05$, $\ctwo=1.05$ & 21304.4\,s & 6.59e-03 \\ \hline
    RA, $\gamma=1$, $\cone=1.10$, $\ctwo=1.10$ & 17942.8\,s & 1.51e-03 \\
    RA, $\gamma=1$, $\cone=1.10$, $\ctwo=1.05$ & 18087.9\,s & 1.33e-03 \\
    RA, $\gamma=1$, $\cone=1.05$, $\ctwo=1.10$ & 18023.9\,s & 1.40e-03 \\
    RA, $\gamma=1$, $\cone=1.05$, $\ctwo=1.05$ & 18194.7\,s & 1.22e-03 
\end{tabular}\caption{\footnotesize TSI. Computational time and average relative error in time \eqref{eq:rel_err_avg} for different rank-adaptive strategies, and comparison with the non-rank-adaptive case. The error indicator for rank adaptivity is computed at all times based on all $p$ test parameters. The FOM is solved in $10391.8\,s$.}\label{tab:TSI_RA_updnozeros_p}
\end{center}
\end{table}

To address this computational bottleneck, we consider the interpolation-based strategy described in \Cref{sec:coeff_upd}. The evolution of the numerical error and the reduced space dimension, and a comparison with the non-rank-adaptive case and the rank-adaptive strategy are shown in \Cref{fig:TSI_RA_updnozeros_pstar}. Finally, average relative errors $\avgsolerr$ and computational runtimes are reported in \Cref{tab:TSI_RA_updnozeros_pstar}. We remark that the error is reduced by a factor of $3.5$ compared to the non-rank-adaptive case, while the runtime is less than three times as much. Moreover, the accuracy of the rank-adaptive model with $\spei=6$ is comparable to that achieved with $\spei=p=100$, with a significant reduction of the computational time.
\begin{figure}[H]
\centering
\includegraphics[]{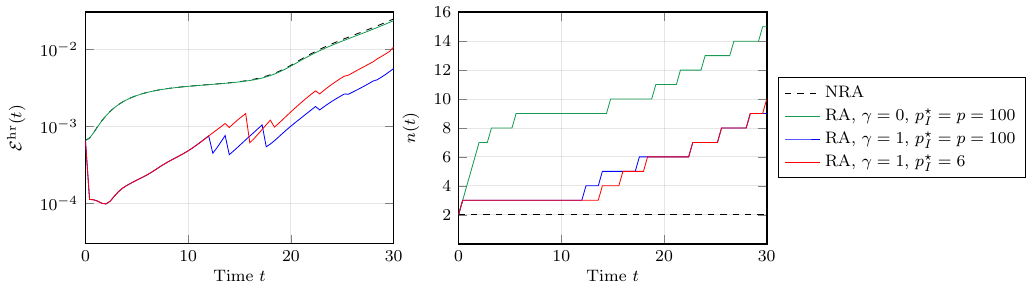}
\caption{\footnotesize TSI. Evolution of the relative error \eqref{eq:rel_err} (left) and of the dimension of the reduced basis (right) for the rank-adaptive and non-rank-adaptive algorithms.
Rank-adaptive system with $\cone=\ctwo=1.05$ in \eqref{eq:update_criterion}. }\label{fig:TSI_RA_updnozeros_pstar}
\end{figure}

\begin{table}[H]
\small
\begin{center}
\begin{tabular}{c|c|c}
    Model & Runtime & $\avgsolerr$ \\ \hline
    NRA & 1115.1\,s & 6.82e-03 \\ \hline
    RA, $\gamma=0$, $\spei=p=100$ & 21304.4\,s & 6.59e-03 \\ \hline
    RA, $\gamma=1$, $\spei=p=100$ & 18194.7\,s & 1.22e-03 \\ \hline
    RA, $\gamma=1$, $\spei=6$ & 3237.3\,s & 2.01e-03
\end{tabular}\caption{\footnotesize TSI. Computational time and average relative error in time \eqref{eq:rel_err_avg} for different rank-adaptive strategies, and comparison with the non-rank-adaptive case. In the rank-adaptive case, the parameters in the update criterion are $\cone=\ctwo=1.05$. The FOM is solved in $10391.8\,s$.}\label{tab:TSI_RA_updnozeros_pstar}
\end{center}
\end{table}

Finally, we show in \Cref{fig:TSI_RA_ham} the evolution of the error $\hamerr$ \eqref{eq:err_ham}. Compared to the NLLD benchmark, the models we considered do not exhibit significant differences in terms of Hamiltonian conservation.
\begin{figure}[H]
\centering
\includegraphics[]{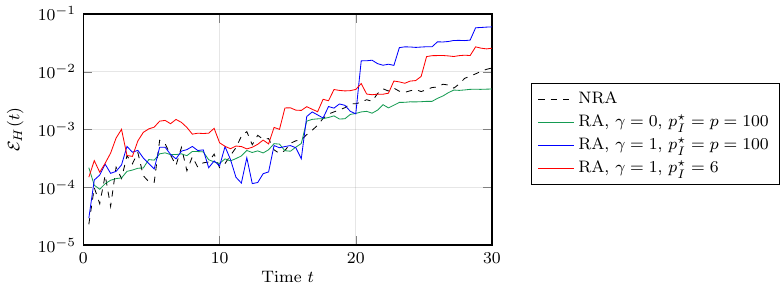}
\caption{\footnotesize TSI. Evolution of the error in the Hamiltonian conservation \eqref{eq:err_ham}. Rank-adaptive system with $\cone=\ctwo=1.05$.}\label{fig:TSI_RA_ham}
\end{figure}

\section{Conclusions, limitations and future work}

We have presented an adaptive hyper-reduction strategy to reduce the computational complexity of evaluating non-sparse nonlinear operators. As a particular application, we have focused on the Hamiltonian system arising from a particle-based discretization of the parametric Vlasov-Poisson equation.
In order to exploit the local low-rank nature of the problem, we have developed an adaptive strategy to update both the reduced space and the EIM hyper-reduction space over time. To this end, we have proposed a parameter sampling technique so that the cost of performing the updates does not depend on the product of the full order dimension and the total number of test parameters. Finally, we have proposed a rank-adaptive procedure to increase the dimension of the reduced basis space in the case of high variations of the rank of the full order solution. The rank adaptation is based on an error indicator that can be constructed efficiently thanks to a suitable interpolation strategy.

A current limitation of the proposed method is that the reducibility properties of the nonlinear term projected onto the reduced space degrade as the reduced space dimension increases. Although our strategy performs best when applied to problems whose solutions are locally low-rank at each time, its performance might deteriorate in rank-adaptive settings with frequent updates of the reduced basis dimension. While the adaptive choice of a tolerance for the construction of the EIM space might be beneficial, this bottleneck can ultimately be addressed by further investigating the interplay between the dimension of the reduced space and the low-rank structure of projected nonlinear operators.
We also mention that in our strategy we construct and evolve one reduced basis for all test parameters. This is useful in situations where a solution has to be computed for a large number of parameters, which are assumed to be given and fixed: in this setting, numerical evidence shows that the hyper-reduced model can achieve speedups of orders of magnitude compared to the full order model. However, this approach is not suitable for scenarios where either the set of parameters is not known a priori or the solution exhibits high variability across the parameter space. In these cases, one could combine the proposed strategy with parameter estimation techniques, or construct reduced bases locally in the parameter space.

Another possible direction for future extensions of this work is the development of different interpolation schemes for the error indicator and the study of their impact on the overall performance of the rank-adaptive model.

\appendix

\section{The case $\pdeg=1$}\label{sec:k1}
In this section we derive the explicit expressions of the constants involved in the error bound of \Cref{prop:error_bound_EIM} in the particular case of a uniform spatial grid $0=x_0<x_1<\dots<x_{N_x}=\ell_x$ with step size $\Delta x$, and piecewise linear polynomials, $\pdeg=1$. This is the scenario adopted in the numerical experiments of \Cref{sec:num_exp}. In this case, the dimension of the finite element space $\mathcal{P}_1(\Omega_x)$ is $\fedim+1=N_x$, and a basis is given by the hat-functions $\{\lambda_i(x)\}_{i=1}^{N_x}$
The vector $\sbf$ introduced in \Cref{sec:semidiscr} is $\sbf=\Delta x\mathbf{1}_\fedim$, and $T$ is the tridiagonal matrix
\begin{equation*}
    T:=(\Delta x)^{-1}\text{tridiag}(-1,2,-1)\in\Rbb^{(N_x-1)\times (N_x-1)}.
\end{equation*}
The eigenpairs $\{(\delta_\idx,\Vbf^\idx)\}_{\idx=1}^{N_x-1}$ of $T$ are \cite{CY00}
\begin{equation*}
    \delta_\idx = 2(\Delta x)^{-1}\left(1+\cos\left(\displaystyle\frac{\idx\pi}{N_x}\right)\right),\qquad
    V^\idx_j=(-1)^j\sqrt{\frac{2}{N_x}}\sin{\left(\frac{j\idx\pi}{N_x}\right)},\quad j=1,\ldots N_x-1.
\end{equation*}
A simple computation shows that
\begin{equation*}
    \sum_{j=1}^{N_x-1}V^\idx_j=\begin{cases}
        0 & N_x+\idx \text{ even} \\
        -\sqrt{\dfrac{2}{N_x}}\sin\left(\dfrac{\idx\pi}{N_x}\right)\left(1+\cos\left(\dfrac{\idx\pi}{N_x}\right)\right)^{-1} & N_x+\idx \text{ odd}
    \end{cases}.
\end{equation*}
Therefore, the terms in the Hamiltonian decomposition \eqref{eq:Ham_dec} are
\begin{equation*}
    \alpha_\idx=\begin{cases}
        0 & N_x+\idx \text{ even} \\
        -\dfrac{\ell_x}{N_x^2}\sqrt{\dfrac{N}{2}}\sin\left(\dfrac{\idx\pi}{N_x}\right)\left(1+\cos\left(\dfrac{\idx\pi}{N_x}\right)\right)^{-3/2} & N_x+\idx \text{ odd}
    \end{cases},
    \quad
    \boldsymbol{\beta}_i^\idx=-\frac{\ell_x}{2\sqrt{NN_x}}\left(1+\cos\left(\dfrac{\idx\pi}{N_x}\right)\right)^{-1/2},
\end{equation*}
for all $i=1,\ldots, N$,
and the constant $K_1$ in the error bound given by \Cref{prop:error_bound_EIM} is
\begin{equation*}
    K_1 = \ell_x N_x^{-1/2}\left(1-\cos\left(\dfrac{\pi}{N_x}\right)\right)^{-1/2}.
\end{equation*}
We observe that this quantity is independent of $N$ and it is proportional to $\sqrt{N_x}$ when $N_x$ is large. Moreover, since
\begin{equation*}
    \max_{x\in\Omega_x}\bigg\lvert\sum_{j=1}^\kappa V_j^\idx\lambda_j^\prime(x)\bigg\rvert=\max_j\frac{\lvert V^{\idx}_{j+1}-V^{\idx}_j\rvert}{\Delta x}=\frac{2}{\ell_x}\sqrt{2N_x}\left\lvert\cos\left(\frac{\idx\pi}{2N_x}\right)\right\rvert,
\end{equation*}
the constant $K_2$ satisfies
\begin{equation*}
    K_2 \leq K_1 N^{-1/2}\max_\idx\norm{((\deimb{\idx})^\top \deimi{\idx})^{-1}(\deimb{\idx})^\top\mathbf{1}_N}\leq K_1\max_\idx\norm{((\deimb{\idx})^\top \deimi{\idx})^{-1}}_2.
\end{equation*}

\printbibliography

\end{document}